\documentclass{amsproc}

\usepackage{amsmath,amsthm,amssymb}
\usepackage{graphicx,color}
\usepackage{texdraw}

\def\ep{\varepsilon}
\def\epsilon{\varepsilon}

\def\RR{{\mathbb R} }

\def\di{\displaystyle}
\def\ri{\rightarrow}

\theoremstyle{plain}
\newtheorem{theorem}{Theorem}[section]
\newtheorem{proposition}[theorem]{Proposition}

\newtheorem{lemma}[theorem]{Lemma}
\theoremstyle{definition}

\newtheorem{remark}[theorem]{Remark}
\newtheorem{definition}[theorem]{Definition}

\def\RR{{\mathbb R} }

\def\ep{\varepsilon}
\def\dist{\text{dist}}

\title[Sub-super solutions for $p$-laplacians]{Sub and supersolutions, invariant cones and multiplicity results for $p$-Laplace equations}
\author{Maria-Magdalena Boureanu }
\thanks{M.-M. Boureanu acknowledges her support by Grant CNCS PCE--47/2011}
\address{Department of Applied Mathematics,
University of Craiova,
200585 Craiova,
Romania}
\email{mmboureanu@yahoo.com}

\author{ Benedetta Noris}
\address{Dipartimento di Matematica e Applicazioni, Universit\`a di Milano-Bicocca, Piazza Ateneo Nuovo 1
20126 Milano }
\email{benedettanoris@gmail.com}

\author{ Susanna Terracini}
\address{Dipartimento di Matematica ``Giuseppe Peano'', Universit\`a ti Torino, Via Carlo Alberto 10, 20123 Torino (Italy)}
\email{susanna.terracini@unito.it}

\thanks{B. Noris and S. Terracini are partially supported by the PRIN2009 grant ``Critical Point Theory and Perturbative Methods for Nonlinear Differential Equations''.}

\begin{document}

\maketitle

\begin{abstract}For a class of quasilinear elliptic equations involving the $p$-Laplace operator,  we develop an abstract  critical point theory in the presence of sub-supersolutions. Our approach is based upon the proof of the invariance under the gradient flow of enlarged cones in the $W^{1,p}_0$ topology. With this, we prove abstract existence and multiplicity theorems in the presence of variously ordered pairs of sub-supersolutions.  As an application, we provide a four solutions theorem, one of the solutions being sign-changing.
  \end{abstract}

\subjclass[2000]{Primary 35J92, 35A01, (35A16, 35B05, 35J20, 58E05) }
\bigskip

    {\small{\bf Keywords:} quasilinear elliptic equation, $p$-Laplace operator, existence, multiplicity, sub-supersolutions, invariance of the cones, sign-changing solution.}

\section{Introduction}\label{section:introduction}

In the present paper we develop a min-max theoretical approach to the sub-supersolution method  in order to obtain general existence results for quasilinear problems of the type
\begin{equation}\label{eq:main_equation}
\left\{\begin{array}{ll}
        -\Delta_p u=f(x,u) \quad\text{in } \Omega \\
    u\in W^{1,p}_0(\Omega),
       \end{array}\right.
\end{equation}
where $\Omega\subset\RR^N$ ($N\geq2$) is a bounded regular domain and $\Delta_p u=\text{div}(|\nabla u|^{p-2}\nabla u)$, $p>1$. The function $f:\Omega\times\RR\to\RR$ satisfies the following assumptions
\begin{itemize}
\item[$(f1)$] there exist $1<q<p^\star$ and positive constants $c_1,c_2$ such that
\[
|f(x,t)|\leq c_1+c_2|t|^{q-1} \qquad \forall t\in\RR, \text{ for a.e. } x\in\Omega;
\]
\item[$(f2)$] $f\in C(\overline{\Omega}\times\RR)$; if $p\geq2$ then $f(x,\cdot)$ is locally Lipschitz continuous, uniformly for $x\in\Omega$; if $p<2$ then $f(x,\cdot)$ is locally $(p-1)$-H\"older continuous, uniformly for $x\in\Omega$.
\item[$(f3)$] there exists $M>0$ such that $h:\Omega\times\RR\to\RR$ given by
\begin{equation}\nonumber
h(x,t)=f(x,t)+M|t|^{p-2}t
\end{equation}
 is nondecreasing in $t$.
\end{itemize}
Here as usual $p^\star=Np/(N-p)$ in case $p<N$ and $p^\star=+\infty$ in case $p\geq N$.

In the classical setting, the sub-supersolution existence theorem
requires the existence of  an ordered pair  $\alpha\leq\beta$ of a bounded subsolution $\alpha$ and a bounded supersolution $\beta$, and states the existence of a solution of the equation in between.  The sub-supersolution method, heavily relying on the maximum principle, was originally used for ODE's and semilinear elliptic equations and then developed over the years into a large variety of techniques, see for example the book \cite{DeCosterHabets2006book} for an exhaustive discussion concerning the wide literature on the topic.


It is not too difficult to adapt this strategy to quasilinear equations \eqref{eq:main_equation} in order to prove the existence of at least one solution. More precisely we will show that, in presence of a pair of ordered strict sub-supersolutions to \eqref{eq:main_equation}, there always exists a (locally) minimal energy solution in the order interval. Here the energy functional is
\begin{equation}\label{eq: our functional J}
J(u)=\int_\Omega \left( \frac{|\nabla u|^p}{p}-F(x,u)\right)\,dx, \qquad u\in W_0^{1,p}(\Omega),
\end{equation}
where $F(x,t)=\int_0^t f(x,s) \,ds$.  In case $p=2$ this result has been first proved by Hofer \cite{Hofer1982} and generalized by various authors in \cite{BrezisNirenberg1993,DeFigueiredoSolimini1984,KCChang1983,KCChang1983variational}.

Concerning multiplicity results, and in particular sign changing solutions, the main interest went to ordinary differential equations and to semilinear elliptic equations.
Amann \cite{Amann1971,Amann1972} combined for the first time the technique of sub-supersolutions with the theory of topological degree, proving the well known three solutions theorem. Ten years later, Hofer \cite{Hofer1982} brought together the variational methods with the topological ones, thus obtaining
multiplicity results and degree properties of the solutions, by working in a partially ordered Hilbert space, that is, a Hilbert space with an ordering given by a closed proper cone. This fruitful perspective has been further deepened and, starting from \cite{DeFigueiredoSolimini1984,KCChang1983,KCChang1983variational}, has finally lead to a Morse theoretical approach. Since then, the method has been further generalized and applied, see e.g. \cite{Bartsch2001,BartschWang1996,ContiMerizziTerracini1999,DancerDu1995,DancerDu1997,Wang1991}.
%

Several difficulties occur when trying to adapt these techniques to the case $p\neq2$, an immediate one being the lack of an underpinning Hilbert structure. The first papers in this direction are those by Bartsch and Liu \cite{BartschLiu2004location,BartschLiu2004multiple,BartschLiu2004}. The authors develop an abstract critical point theory in partially ordered Banach spaces and provide several applications, such as the existence of sign changing solutions to \eqref{eq:main_equation} under suitable assumptions. Bartsch and Liu consider the operator $K:W_0^{1,p}(\Omega)\to W_0^{1,p}(\Omega)$ defined as
\begin{equation}\label{eq:definition_of_K}
v=K(u) \qquad \text{ is the solution of }\qquad -\Delta_p v+M|v|^{p-2}v=h(x,u)\quad\text{in } \Omega
\end{equation}
and show that $u-K(u)$ serves as a pseudogradient vector field for $J'(u)$. This property is a consequence of Simon's inequalities  \cite{Simon1978}. One of the main difficulties is the fact that  map $u\mapsto u-K(u)$ is no longer a Lipschitz pseudo gradient vector field, therefore, in order to apply the standard variational techniques, the authors need to construct a Lipschitz one  which still satisfies  Simon's inequalities (see also \cite{LiuSun2001}). We point out that Bartsch and Liu work in the space of $C^1$ functions since, to prove the existence of sign-changing solutions, they consider the cones of positive and negative functions, which have empty interior in the space $W_0^{1,p}$.

In the subsequent paper \cite{BartschLiuWeth2005}, Bartsch, Liu and Weth work in the Sobolev space $W_0^{1,p}(\Omega)$, by adopting a strategy which was already introduced in \cite{ContiMerizziTerracini1999} for the case of semilinear equations. The strategy consists in showing that an $W^{1,p}$- open neighborhood of the cone of positive functions is invariant under the action of the flow associated to the Lipschitz pseudogradient vector field.

Our first goal is to extend the invariance result to enlargements of cones generated by not necessarily constant sub-supersolutions (Theorem \ref{theorem:invariance_of_the_cone}). This  is a  non trivial property in the framework of  quasilinear equations, the difficulty being related to the lack of a general strong comparison principle for equation \eqref{eq:main_equation}. Indeed, in order to prove the invariance of an open neighborhood of the cone generated by a subsolution, we will need to deal with a strict subsolution, in the sense that it satisfies the equation up to a strictly negative remainder. In addition, depending on the values of $p$ and $N$, we will impose some integrability conditions on this remainder.

As an application of the abstract invariance theorem, we prove the existence of a sign changing solution to \eqref{eq:main_equation}, under suitable additional assumptions on $f$, and when $p>({N-2+\sqrt{9N^2-4N+4}})/({2N})$. This result generalizes to quasilinear equations the four solutions theorem of \cite[Theorem 2]{BartschWang1996} and \cite[Theorem 3.7]{DeCosterHabets2006book}. Note in particular that, unlike most of the related results, we do not impose an ordering between $p$ and $q$ in assumption $(f1)$ and we do not impose Ambrosetti-Rabinowitz type conditions.


\section{Statement of the main results}

%
In what follows equations and inequalities are always intended in the weak sense. As usual, we say that $\alpha\in W^{1,p}(\Omega)$ is a subsolution for \eqref{eq:main_equation} if
\begin{equation*}
\left\{\begin{array}{ll}
        -\Delta_p \alpha\leq f(x,\alpha) &\quad\text{in } \Omega \\
        \alpha\leq 0 &\quad\text{on } \partial\Omega.
       \end{array}\right.
\end{equation*}
Analogously, we say that $\beta\in W^{1,p}(\Omega)$ is a supersolution if the opposite inequalities hold. Following \cite{ContiMerizziTerracini1999} we also introduce a more restrictive notion.

\begin{definition}\label{definition_subsolution_supersolution}
We say that $\alpha\in W^{1,p}(\Omega)\cap L^\infty(\Omega)$ is a strict subsolution for \eqref{eq:main_equation} if there exists $a\in L^p(\Omega)$, with $a(x)>0$ for a.e. $x\in\Omega$, such that  it holds
\begin{equation}\nonumber
\left\{\begin{array}{ll}
        -\Delta_p \alpha=f(x,\alpha)-a(x) &\quad\text{in } \Omega \\
        \alpha\leq 0 &\quad\text{on } \partial\Omega.
       \end{array}\right.
\end{equation}
In a similar way we say that $\beta\in W^{1,p}(\Omega)\cap L^\infty(\Omega)$ is a strict supersolution for \eqref{eq:main_equation} if there exists $b\in L^p(\Omega)$, $b(x)>0$ for a.e. $x\in\Omega$, such that
\begin{equation}\nonumber
\left\{\begin{array}{ll}
        -\Delta_p \beta=f(x,\beta)+b(x) &\quad\text{in } \Omega \\
        \beta\geq 0 &\quad\text{on } \partial\Omega.
       \end{array}\right.
\end{equation}
\end{definition}
Given a subsolution $\alpha\in W^{1,p}(\Omega)$ for \eqref{eq:main_equation}, we define the positive cone with vertex in $\alpha$ as
\[
{\mathcal C}_{\alpha}:=\{ u\in W_0^{1,p}(\Omega):\ u(x)\geq\alpha(x) \text{ for a.e. }x\in\Omega  \}.
\]
Analogously, given a supersolution $\beta\in W^{1,p}(\Omega)$, the negative cone with vertex in $\beta$ is
\[
{\mathcal C}^{\beta}:=\{ u\in W_0^{1,p}(\Omega):\ u(x)\leq\beta(x) \text{ for a.e. }x\in\Omega  \}.
\]
First we prove that, given an ordered couple $\alpha<\beta$, there always exists a solution of \eqref{eq:main_equation}, which is of minimal energy between $\alpha$ and $\beta$. Note that, even if $\alpha$ and $\beta$ are nonconstant functions, by $\alpha<\beta$ we naturally understand $\alpha(x)<\beta(x)$ for a.e. $x\in\Omega$.
\begin{theorem}\label{theorem:order}
Let $f$ satisfy $(f1)-(f3)$ and assume that there exist a subsolution $\alpha\in W^{1,p}(\Omega)$ and a supersolution $\beta\in W^{1,p}(\Omega)$ for \eqref{eq:main_equation} such that $\alpha(x)<\beta(x)$ for a.e. $x\in\Omega$. Then there exists $u_0\in {\mathcal C}_{\alpha}\cap {\mathcal C}^{\beta}$, solution of \eqref{eq:main_equation}, which satisfies
\[
J(u_0)=\min_{u\in {\mathcal C}_{\alpha}\cap {\mathcal C}^{\beta}} J(u).
\]
\end{theorem}
Section \ref{section:proof of Theorem 1.2} is dedicated to the proof of this result which generalizes Proposition 1 by Hofer \cite{Hofer1982} and Theorem 6 by De Figueiredo and Solimini \cite{DeFigueiredoSolimini1984} to quasilinear equations. The main difficulties in adapting these proofs are that $W_0^{1,p}(\Omega)$ is not a Hilbert space and that it is not possible to associate to $J'(u)$ a gradient belonging to $W_0^{1,p}(\Omega)$. Instead, we will take advantage of the fact that ${\mathcal C}_{\alpha}$ and ${\mathcal C}^{\beta}$ are invariant under the action of the operator $K$ defined in \eqref{eq:definition_of_K}. A similar result was obtained in \cite[Theorem 2.1]{BartschLiu2004location} and \cite[Proposition 3.2]{PapageorgiouRochaStaicu2008}, where the nonlinearity $f$ is assumed to have a $p$-superlinear growth at infinity by imposing an Ambrosetti-Rabinowitz condition. Notice that, since we do not require the Ambrosetti-Rabinowitz condition on $f$, $J$ needs not to satisfy the Palais-Smale condition in the entire space. Nonetheless, we will show that $J$ satisfies the Palais-Smale condition in the set ${\mathcal C}_{\alpha}\cap {\mathcal C}^{\beta}$, which is sufficient to prove the existence of the solution $u_0$.

To prove multiplicity results by means of variational methods, we will consider open neighborhoods of the cones ${\mathcal C}_{\alpha}$ and $C^\beta$. To this aim we define, given any $\epsilon>0$,
\begin{gather*}
{\mathcal C}_{\alpha,\ep}=\{ u\in W_0^{1,p}(\Omega): \ \dist(u,{\mathcal C}_{\alpha})<\ep \}, \quad
{\mathcal C}^{\beta,\ep}=\{ u\in W_0^{1,p}(\Omega): \ \dist(u,{\mathcal C}^\beta)<\ep \},
\end{gather*}
where, for every $u\in W_0^{1,p}(\Omega)$, we set
\[
\dist(u,{\mathcal C}):=\inf_{w\in {\mathcal C}}\left(\int_\Omega|\nabla (u-w)|^p\,dx\right)^{1/p}.
\]
Next, we recall another definition from \cite{ContiMerizziTerracini1999}.
\begin{definition}\label{def:strict_K_invariance}
Given a strict subsolution $\alpha$ for \eqref{eq:main_equation}, we say that ${\mathcal C}_{\alpha}$ is strictly $K$-invariant if there exists $\epsilon_{\alpha}$ such that
\[
K({\mathcal C}_{\alpha,\ep}) \subseteq {\mathcal C}_{\alpha,\ep/2} \quad\text{ for all } 0<\epsilon<\epsilon_{\alpha}.
\]
Analogously, given a strict supersolution $\beta$, ${\mathcal C}^{\beta}$ is strictly $K$-invariant if there exists $\epsilon_{\beta}$ such that $K({\mathcal C}^{\beta,\ep})\subseteq {\mathcal C}^{\beta,\ep/2}$ for all $0<\epsilon<\epsilon_{\beta}$.
\end{definition}
Furthermore, we also need a weaker notion of $K$-invariance in the following sense.
\begin{definition}\label{def:locally_K_invariance}
Given a strict subsolution $\alpha$ for \eqref{eq:main_equation}, we say that ${\mathcal C}_{\alpha}$ is locally $K$-invariant if for every bounded subset $\mathcal{U}\subset W_0^{1,p}(\Omega)$  there exists $\epsilon_{\alpha}$ depending on $\mathcal{U}$ such that
\[
K({\mathcal C}_{\alpha,\ep}\cap\mathcal{U}) \subseteq {\mathcal C}_{\alpha,\ep/2} \quad\text{ for all } 0<\epsilon<\epsilon_{\alpha}.
\]
Analogously, given a strict supersolution $\beta$, ${\mathcal C}^{\beta}$ is locally $K$-invariant if for every bounded subset $\mathcal{U}\subset W_0^{1,p}(\Omega)$ there exists $\epsilon_{\beta}$ depending on $\mathcal{U}$ such that $K({\mathcal C}^{\beta,\ep}\cap\mathcal{U})\subseteq {\mathcal C}^{\beta,\ep/2}$ for all $0<\epsilon<\epsilon_{\beta}$.
\end{definition}

We state now our abstract multiplicity result and we prove it in Section \ref{section:not ordered sub-supersolutions}.
%
\begin{theorem}\label{theorem:not ordered}
Let $f$ satisfy $(f1)-(f3)$ and assume that there exist two strict
subsolutions $\alpha_1$, $\alpha_2$ and two strict supersolutions
$\beta_1$, $\beta_2$ for \eqref{eq:main_equation} such that ${\mathcal{C}}_{\alpha_1}$, ${\mathcal{C}}_{\alpha_2}$, ${\mathcal{C}}^{\beta_1}$ and ${\mathcal{C}}^{\beta_2}$ are locally $K$-invariant. Moreover, assume that $\alpha_1$, $\alpha_2$, $\beta_1$, $\beta_2$ are fulfilling
\[
\alpha_1(x)<\beta_1(x), \quad \alpha_2(x)<\beta_2(x)\quad \mbox{for a.e. }x\in\Omega\]
and
\[
\beta_1(x)<\alpha_2(x) \quad \mbox{for $x$ in a set of positive measure}.
\]
Then there exist three different solutions $u_1$, $u_2$, $u_3$ to problem
\eqref{eq:main_equation} satisfying
\[
u_1 \in {\mathcal{C}}_{\alpha_1} \cap {\mathcal{C}}^{\beta_1}, \quad
u_2 \in {\mathcal{C}}_{\alpha_2} \cap {\mathcal{C}}^{\beta_2}
\]
and
\[
u_3\in ({\mathcal{C}}_{\alpha_1} \cap {\mathcal{C}}^{\beta_2})\setminus({\mathcal{C}}_{\alpha_2} \cup {\mathcal{C}}^{\beta_1}).
\]
\end{theorem}

In order to apply the previous theorem, we provide below more explicit conditions on the sub-supersolutions, which ensure that the corresponding cones are locally $K$-invariant.
\begin{theorem}\label{theorem:invariance_of_the_cone}
Let $f$ satisfy $(f1)-(f3)$. Let $\alpha$ be a strict subsolution and $\beta$ be a strict supersolution for \eqref{eq:main_equation}, with remainders $a,b$ respectively, given in Definition \ref{definition_subsolution_supersolution}. Then
\begin{itemize}
\item[(i)]  ${\mathcal C}_{\alpha}$ and ${\mathcal C}^{\beta}$ are locally $K$-invariant if $2N/(N+2)\leq p<2$ (the first inequality being strict for $N=2$) and
\begin{equation}\label{eq:integrability_of_1/a_p<2}
\left(\frac{1}{a}\right)^{\frac{2-p}{p-1}\frac{p^\star}{p^\star-2}},\left(\frac{1}{b}\right)^{\frac{2-p}{p-1}\frac{p^\star}{p^\star-2}} \in L^{1}(\Omega);
\end{equation}
\item[(ii)]  ${\mathcal C}_{\alpha}$ and ${\mathcal C}^{\beta}$ are strictly $K$-invariant if either $p=2$, or $p>2$ and
\begin{equation}\label{eq:integrability_of_1/a}
\frac{1}{a},\frac{1}{b} \in L^r(\Omega) \qquad \text{with} \qquad r\left\{\begin{array}{ll}
=(p-2)\frac{N}{p} \quad&\text{ if } 2< p< N, \\
>p-2 &\text{ if } p=N, \\
=p-2 &\text{ if } p>N.
\end{array}\right.
\end{equation}
\end{itemize}
\end{theorem}
Note that, in case $p=2N/(N+2)$, equation \eqref{eq:integrability_of_1/a_p<2} is to be understood as $1/a,1/b \in L^\infty(\Omega)$.

Since we expect Theorem \ref{theorem:not ordered} to have different applications, in addition to Section \ref{section: K invariance} where we establish the above result, we add Section \ref{section: comments}  where we will investigate integrability conditions different from \eqref{eq:integrability_of_1/a_p<2} and \eqref{eq:integrability_of_1/a} which in may be less restrictive, depending on the situation. Such conditions will depend on the growth of $f$ at infinity and on the dimension $N$.

As an application to Theorem \ref{theorem:not ordered} and Theorem \ref{theorem:invariance_of_the_cone}, we consider assumptions on $f$ and $p$ which ensure that problem \eqref{eq:main_equation} admits a sign changing solution. More precisely we consider the following hypotheses on $f$:
\begin{itemize}
\item[$(f4)$] there exist $0<\mu<\lambda_1$ and $R>0$ such that
\[
\frac{f(x,t)}{|t|^{p-2}t}\leq\mu \quad\text{ for every } |t|>R \text{ and a.e. } x\in\Omega;
\]
\item[$(f5)$] there exists $\lambda_2<\lambda<\infty$ such that
\[
\lim_{t\to 0}\frac{f(x,t)}{|t|^{p-2}t}=\lambda \quad\text{ uniformly for a.e. } x\in\Omega.
\]
Moreover, if $p>2$, there exists a small neighborhood  ${\mathcal V}$ of $t=0$ such that $f(x,\cdot)$ is differentiable in ${\mathcal V}$.
\end{itemize}
Here, as usual,
\[
\lambda_2=\min\{ \lambda>\lambda_1: \text{ there exists } \phi \in W_0^{1,p}(\Omega), \phi\not\equiv0 \text{ such that } -\Delta_p\phi=\lambda|\phi|^{p-2}\phi \}
\]
and
\begin{equation}\nonumber
\lambda_1 = \inf_{\substack{\varphi \in W_0^{1,p}(\Omega) \\ \varphi\not\equiv0}}
\frac{\int_\Omega |\nabla \varphi|^p\, dx}{\int_\Omega |\varphi|^p\, dx}.
\end{equation}

Notice that conditions $(f2)$ and $(f4)$ imply condition $(f1)$ and we provide the following result that will be proved in Section \ref{section: application}.
\begin{theorem}\label{theorem:application}
Let $p>({N-2+\sqrt{9N^2-4N+4}})/({2N})$ and let $f$ satisfy $(f2)-(f5)$. Then, in addition to the trivial solution, there exist a positive solution, a negative solution and a sign changing solution to problem \eqref{eq:main_equation}.
\end{theorem}
We will prove this way that the four solutions theorem, known for $p=2$ (see \cite[Theorem 2]{BartschWang1996}), holds for a larger range of the parameter $p$ and we emphasize the fact that we consider both cases $p<2$ and $p>2$. At the same time, our previous theorem generalizes some results in \cite{BartschLiu2004,BartschLiuWeth2005,PapageorgiouRochaStaicu2008}, where the sub-supersolutions are considered to be constant. Hence, an important improvement provided by our study is that we deal with nonconstant sub-supersolutions when treating a quasilinear problem.

\section{Preliminaries}\label{section:preliminaries}
In this section we introduce both a variational and a fixed point framework for problem \eqref{eq:main_equation}, together with the related known results that we will use in the next sections. We will tacitly assume $(f1)-(f3)$. Also, unless otherwise stated, everywhere in this paper $C$ denotes a generic constant that may change its value from line to line.

We will work in the space $W_0^{1,p}(\Omega)$ endowed with the norm
\[
\|u\|=\left(\int_\Omega|\nabla u|^p\,dx\right)^{1/p}.
\]
%
We will denote by $W^{-1,p'}(\Omega)$ the dual of $W_0^{1,p}(\Omega)$, where, as usual, $1/p+1/p'=1$.

In our search for weak solutions to problem \eqref{eq:main_equation} we are relying on the critical point theory. We associate to our problem the energetic functional $J$ introduced in \eqref{eq: our functional J}. By a standard calculus we can establish that $J\in C^1(W_0^{1,p}(\Omega);\RR)$ (it is worth to notice that $J$ is not of class $C^2$ in case $p<2$) and its G\^{a}teaux derivative is given by the formula
\begin{equation*}
J'(u)[v]=\int_\Omega|\nabla u|^{p-2}\nabla u\nabla v\,dx -
\int_\Omega f(x,u)v\,dx\quad \forall u,\, v\in W_0^{1,p}(\Omega),
\end{equation*}
hence the critical points of $J$ are in fact weak solutions to problem \eqref{eq:main_equation}. Then it is only natural to focus on the properties of $J$. A fundamental tool in proving these properties is represented by the inequalities listed in the three lemmas below.
\begin{lemma}\label{lemma:inequalities_for_vectors}
For every $\xi,\eta\in\RR^N$ it holds
\begin{gather}
(|\xi|^{p-2}\xi-|\eta|^{p-2}\eta)\cdot(\xi-\eta)\geq (p-1)|\xi-\eta|^2(|\xi|+|\eta|)^{p-2} \quad\text{ if }\quad 1<p\leq 2 \nonumber
\\
(|\xi|^{p-2}\xi-|\eta|^{p-2}\eta)\cdot(\xi-\eta)\geq 2^{2-p} |\xi-\eta|^{p} \quad\text{ if }\quad p\geq 2 \label{eq:simon2} \\
||\xi|^{p-2}\xi-|\eta|^{p-2}\eta| \leq c_p |\xi-\eta|^{p-1} \quad\text{ if }\quad 1<p\leq 2 \nonumber \\
||\xi|^{p-2}\xi-|\eta|^{p-2}\eta| \leq (p-1) (|\xi|+|\eta|)^{p-2}|\xi-\eta| \quad\text{ if }\quad p\geq 2 \nonumber
\end{gather}
for some constant $c_p>0$.
\end{lemma}
Note that the first two inequalities from the above lemma appeared for the first time in \cite{GlowinskiMarrocco1975} for the case $N=2$ and then in \cite{Simon1978} for any dimension, while the remaining ones can be found in the more recent \cite{Damascelli1998} (see also \cite[Section 10]{lindqvist2006notes} for the proofs). By combining such relations with suitable H\"older inequalities it is possible to prove the following.
\begin{lemma}(\cite[Lemma 2.1]{Simon1978})\label{lemma:inequalities_for_integrals_appendix_1}
Let $u,v\in L^p(\Omega)$ for some $1<p\leq 2$. Then
\[
\displaystyle{ \int_\Omega(|u|^{p-2}u-|v|^{p-2}v)(u-v)\,dx\geq (p-1)\||u|+|v|\|_{L^p(\Omega)}^{p-2} \|u-v\|_{L^p(\Omega)}^2 }.
\]
\end{lemma}
\begin{lemma}(\cite[Lemma 3.8]{BartschLiu2004})\label{lemma:inequalities_for_integrals_appendix_2}
Let $u,v\in L^p(\Omega)$ for some $p\geq 2$. Then
\[
\||u|^{p-2}u-|v|^{p-2}v\|_{L^{p'}(\Omega)} \leq (p-1) \||u|+|v|\|_{L^p(\Omega)}^{p-2} \|u-v\|_{L^p(\Omega)}.
\]
\end{lemma}

As a consequence of Lemma \ref{lemma:inequalities_for_integrals_appendix_1} and of \eqref{eq:simon2}, there exists $C>0$
such that the following holds for every $u,v\in W_0^{1,p}(\Omega)$
\begin{equation}\label{eq:inequality_appearing_several_times}
\int_\Omega(|\nabla u|^{p-2}\nabla u-|\nabla v|^{p-2}\nabla v)\cdot\nabla(u-v)\, dx \geq
\left\{\begin{array}{ll}
C(\|u\|+\| v \|)^{p-2}\|u-v\|^2  & \text{ if } p\leq 2\\
C \|u-v\|^p  & \text{ if }p\geq 2.
\end{array}\right.
\end{equation}
This relation plays a key role in the proofs of the subsequent results, such as the well known compactness property below (see for example \cite[Appendix A]{Peral1997}).
\begin{lemma}\label{lemma:palais_smale_if_u_n_bounded}
Let $(u_n)_n\subset W_0^{1,p}(\Omega)$ be a bounded sequence such that\\ $\|J'(u_n)\|_{W^{-1,p'}(\Omega)}  \to 0$. Then there exists $u\in W_0^{1,p}(\Omega)$ such that $u_n\to u$ in $W_0^{1,p}(\Omega)$ and $J'(u)=0$.
\end{lemma}

As for the fixed point framework, let us recall some fundamental properties of the operator $K(u)$ introduced in \eqref{eq:definition_of_K}.
We can examine the mutual relations between $J(u)$ and $K(u)$. Despite the fact that $J'(u)$ does not admit in general a representative in the space $W_0^{1,p}(\Omega)$, as is the case when $p=2$, some useful relations can be proved. More exactly, due to \cite[relation (2.2)]{Simon1978}, we have the following lemma.
\begin{lemma}\label{lemma:K_pseudogradient_first_estimate}
There exists $C>0$ such that for every $u\in W_0^{1,p}(\Omega)$ the following holds
\[
J'(u)[u-K(u)]\geq \left\{\begin{array}{ll}
              C\|u-K(u)\|^2 (\|u\|+\|K(u)\|)^{p-2} \quad &\text{ if } 1<p\leq 2 \\
              C\|u-K(u)\|^p  \quad &\text{ if } p\geq 2.
                         \end{array}\right.
\]
\end{lemma}
In addition, due to \cite[Lemma 3.8]{BartschLiu2004}, we are able to give more estimates.
\begin{lemma}\label{lemma:K_pseudogradient_second_estimate}
There exists $C>0$ such that for every $u\in W_0^{1,p}(\Omega)$ the following holds
\[
\|J'(u)\|_{W^{-1,p'}(\Omega)} \leq \left\{\begin{array}{ll}
              C\|u-K(u)\|^{p-1} \quad &\text{ if } 1<p\leq 2 \\
              C\|u-K(u)\| (\|u\|+\|K(u)\|)^{p-2} \quad &\text{ if } p\geq 2.
                         \end{array}\right.
\]
\end{lemma}
\begin{remark}\label{rem:bounded norm}
It is worth keeping in mind that, by assumption $(f1)$, $\|K(u)\|$ is bounded whenever $\|u\|$ is.
\end{remark}
Some attention should be paid at this point because the operator $u-K(u)$ is not Lipschitz for $p\neq 2$, so that it can not be used as a generalized pseudogradient vector field for $J'(u)$. To overcome this obstacle, we rely on \cite[Lemma 4.1]{BartschLiu2004} and \cite[Lemma 2.1]{BartschLiuWeth2005} which allow us to formulate the next proposition.
\begin{proposition}\label{prop:lipschitz_pseudogradient}
Let $\alpha$ be a strict subsolution for \eqref{eq:main_equation} such that $\mathcal{C}_\alpha$ is strictly (respectively locally) $K$-invariant. Then there exists a locally Lipschitz continuous operator $\tilde K: W_0^{1,p}(\Omega)\setminus\{u: u=K(u)\}\to W_0^{1,p}(\Omega)$ satisfying the inequalities from Lemmas \ref{lemma:K_pseudogradient_first_estimate} and \ref{lemma:K_pseudogradient_second_estimate} such that $\mathcal{C}_\alpha$ is strictly (respectively locally) $\tilde K$-invariant. An analogous result holds for a strict supersolution $\beta$.
\end{proposition}
%
%
To conclude this section, let us recall some known properties of $\lambda_1$ and $\lambda_2$ and of the associated eigenfunctions (see \cite{DiBenedetto1983,Lieberman1988,Lindqvist1990,Tolksdorf1984,Vazquez1984}).
\begin{proposition}\label{prop: first eigenvalue and eigenfuntion}
There is a first eigenfunction $\phi_1 \in C^1(\overline{\Omega})$ corresponding to $\lambda_1$. Moreover, it is simple and (by eventually taking its modulus) we have $\phi_1>0$ in $\Omega$ and $\partial \phi_1/\partial\nu<0$ on $\partial\Omega$.
\end{proposition}
Concerning $\lambda_2$, we will need the following equivalent characterization.
\begin{proposition}(\cite[Corollary 3.2]{CuestaDeFigueiredoGossez1999})\label{prop: second eigenvalue}
Let $\Gamma=\{ \gamma\in C([0,1],W_0^{1,p}(\Omega)): \ \int_\Omega |\gamma(s)|^p\,dx=1, \ s\in [0,1],\ \gamma(0)=-\phi_1,\ \gamma(1)=\phi_1 \}$. Then
\[
\lambda_2=\inf_{\gamma\in\Gamma} \max_{u\in\gamma([0,1])} \int_\Omega |\nabla u|^p\,dx.
\]
\end{proposition}


\section{Minimal energy solution between ordered sub-supersolutions}\label{section:proof of Theorem 1.2}

In this section we prove Theorem \ref{theorem:order}. For this reason we will assume throughout the section that $f$ satisfies $(f1)-(f3)$ and that a subsolution $\alpha$ and a supersolution $\beta$ are given, such that $\alpha(x)<\beta(x)$ for a.e. $x\in\Omega$.

Let us start by recalling the properties of the distance of a point from a convex set, together with the notation of projection that we will use ahead.
\begin{lemma}\label{lemma:properties_of_distance}
Let $\alpha$ be a subsolution for \eqref{eq:main_equation}. Given any $u\in W_0^{1,p}(\Omega)$, we have
\begin{itemize}
\item[(i)] there exists a unique $\pi_\alpha(u)\in {\mathcal C}_{\alpha}$ satisfying $\|u-\pi_\alpha(u)\|=\dist(u,{\mathcal C}_{\alpha})$;
\item[(ii)] $\dist(u,{\mathcal C}_{\alpha})\leq \|[u-\alpha]^-\|$, where, as usual, $v^+=\max\{v,0\}$ and $v^-=\max\{-v,0\}$;
\item[(iii)] for every $1<s<p^\star$ it holds $\|[u-\alpha]^-\|_{L^s(\Omega)}\leq C \dist(u,{\mathcal C}_{\alpha})$, where $C$ is the Sobolev constant of the embedding $W_0^{1,p}(\Omega)\hookrightarrow L^s(\Omega)$.
\end{itemize}
\end{lemma}
\begin{proof}
Being ${\mathcal C}_{\alpha}$ a closed convex set, property (i) follows. To prove (ii) it is enough to notice that we can choose $w=\alpha+[u-\alpha]^+\in {\mathcal C}_{\alpha}$ in the definition of distance of $u$ from ${\mathcal C}_{\alpha}$. Finally, for every $w\in {\mathcal C}_{\alpha}$ we have $u-w\leq u-\alpha$ and therefore
\[
\|[u-\alpha]^-\|_{L^s(\Omega)}\leq \inf_{w\in {\mathcal C}_{\alpha}} \|[u-w]^-\|_{L^s(\Omega)}\leq
C \inf_{w\in {\mathcal C}_{\alpha}} \|u-w\| = C \dist(u,{\mathcal C}_{\alpha}),
\]
and the lemma is proved.
\end{proof}
\begin{remark}
Naturally, an analogous result is valid for ${\mathcal C}^{\beta}$ and we denote by $\pi^\beta(u)$ the unique element in ${\mathcal C}^{\beta}$ satisfying $\|u-\pi^\beta(u)\|=\dist(u,{\mathcal C}^{\beta})$. To avoid repeating the same arguments, in what follows we will only focus on the properties concerning $\alpha$.
\end{remark}
Let us show that $J$ satisfies the Palais-Smale condition in the intersection of the two cones without assuming an Ambrosetti-Rabinowitz condition on $f$. We prove the result directly for the slightly larger set ${\mathcal C}_{\alpha,\ep}\cap {\mathcal C}^{\beta, \ep}$, since we will need it in the subsequent sections.
\begin{lemma}\label{lemma:boundedness}
Fix an arbitrary $\ep\geq 0$. Then
\begin{itemize}
\item[(i)]
for every $1<s< p^\star$ there exists a constant $C=C(\alpha,\beta,\ep,s)>0$ such that
\[
\|u\|_{L^s(\Omega)}\leq C \quad\forall u\in {\mathcal C}_{\alpha,\ep}\cap {\mathcal C}^{\beta, \ep}.
\]
\item[(ii)] $J$ is bounded below in ${\mathcal C}_{\alpha,\ep}\cap {\mathcal C}^{\beta, \ep}$.
\item [(iii)] if $J$ is bounded in a subset $\mathcal{U}\subset{\mathcal C}_{\alpha,\ep}\cap {\mathcal C}^{\beta, \ep}$, then there exists $C>0$ such that
\[
\|u\|+\|K(u)\|\leq C \quad\forall u\in \mathcal{U}.
\]
\item[(iv)] if $J$ is bounded in a subset $\mathcal{U}\subset{\mathcal C}_{\alpha,\ep}\cap {\mathcal C}^{\beta, \ep}$, then there exists $C>0$ such that for every $u\in \mathcal{U}$,
\begin{gather*}
C \|u-K(u)\|^{p-1} \leq \|J'(u)\|_{W^{-1,p'}(\Omega)} \leq C \|u-K(u)\| \quad\mbox{when } p\geq2, \\
C \|u-K(u)\| \leq \|J'(u)\|_{W^{-1,p'}(\Omega)} \leq C \|u-K(u)\|^{p-1} \quad\mbox{when } p\leq2.
\end{gather*}
\item[(v)] $J$ satisfies the Palais-Smale condition in ${\mathcal C}_{\alpha,\ep}\cap {\mathcal C}^{\beta, \ep}$, that is, if $(u_n)_n\subseteq {\mathcal C}_{\alpha,\ep}\cap {\mathcal C}^{\beta, \ep}$ with $J(u_n)\ri c_0\in\RR$ and $J'(u_n)\ri 0$ in $W^{-1,p'}(\Omega)$, then $u_n\ri u$ in $W_0^{1,p}(\Omega)$ and $J'(u)=0$.
\end{itemize}
\end{lemma}
\begin{proof}
(i) Note that
\[
u\geq \alpha + u-\pi_\alpha(u)\quad\mbox{and}\quad u\leq \beta +u-\pi^\beta(u)
\]
thus
\[
|u|\leq |\alpha|+|u-\pi_\alpha(u)|+|\beta|+|u-\pi^\beta(u)|.
\]
Since $u\in {\mathcal C}_{\alpha,\ep}\cap {\mathcal C}^{\beta, \ep}$, by the Sobolev embedding we infer that for every $1<s< p^\star$ there exists a positive constant $C$ such that
\[\begin{split}
\|u\|_{L^s(\Omega)}&\leq \|\alpha\|_{L^s(\Omega)}+C\| u-\pi_\alpha(u) \|+\|\beta\|_{L^s(\Omega)}+
C\| u-\pi^\beta(u)\|\\
& \leq C(\|\alpha\|_{L^s(\Omega)}+\|\beta\|_{L^s(\Omega)}+2\epsilon).
\end{split}
\]

(ii) By $(f1)$ and the previous estimates, if $u\in {\mathcal C}_{\alpha,\ep}\cap C^{\beta, \ep}$, then
\begin{equation}\label{eq:J bounded below}
\frac{1}{p}\|u\|^p\leq J(u)+C,
\end{equation}
where $C$ is a positive constant.

(iii) By Remark \ref{rem:bounded norm} and relation \eqref{eq:J bounded below} we infer that if $J(u)$ is bounded in $\mathcal{U}$, then $\|u\|+\|K(u)\|$ is also bounded in $\mathcal{U}$.

(iv) We conclude this point by taking into consideration what we just proved and Lemmas \ref{lemma:K_pseudogradient_first_estimate} and \ref{lemma:K_pseudogradient_second_estimate}.

(v) We deduce from point (iii) above that $\{u_n\}$ is bounded and therefore Lemma \ref{lemma:palais_smale_if_u_n_bounded} applies, providing the thesis.
\end{proof}
Next, let us prove that ${\mathcal C}_{\alpha}$ and ${\mathcal C}_{\beta}$ are invariant under the action of $K$. This is due to assumption $(f3)$ and to the weak comparison principle (see \cite[Lemma 3.1]{Tolksdorf1983} and \cite[Theorem 1.2]{Damascelli1998}).
\begin{lemma}\label{lemma:comparison_principle}
Given any subsolution $\alpha\in W_0^{1,p}(\Omega)$ and any supersolution $\beta\in W_0^{1,p}(\Omega)$ for equation \eqref{eq:main_equation} it holds $K({\mathcal C}_{\alpha})\subseteq {\mathcal C}_{\alpha}$ and $K({\mathcal C}_{\beta})\subseteq {\mathcal C}_{\beta}$.
\end{lemma}
\begin{proof}
We shall prove the result for ${\mathcal C}_{\alpha}$. Given $u\in {\mathcal C}_{\alpha}$, we need to prove that $K(u)\geq\alpha$. By $(f3)$ we have
\[
\left\{\begin{array}{ll}
-\Delta_p\alpha-(-\Delta_p K(u))+M(|\alpha|^{p-2}\alpha-|K(u)|^{p-2}K(u))\leq0 \quad &\text{in }\Omega\\
\alpha-K(u)\leq0 &\text{on }\partial\Omega.
\end{array}\right.
\]
By testing with $[\alpha-K(u)]^+$, and recalling that $(|\xi|^{p-2}\xi-|\eta|^{p-2}\eta)\cdot(\xi-\eta)\geq 0$ by Lemma \ref{lemma:inequalities_for_vectors}, we obtain
\[
\int_\Omega \left(|\nabla\alpha|^{p-2}\nabla\alpha-|\nabla K(u)|^{p-2}\nabla K(u)\right)\cdot \nabla [\alpha-K(u)]^+ \,dx \leq 0.
\]
Finally, relation \eqref{eq:inequality_appearing_several_times} implies $[\alpha-K(u)]^+\equiv 0$, which concludes the proof.
\end{proof}
We have all the tools to prove Theorem \ref{theorem:order}. We follow the approach of De Figueiredo and Solimini \cite{DeFigueiredoSolimini1984}.
\begin{proof}[Proof of Theorem \ref{theorem:order}]
By Ekeland's variational principle there exists \\ $(u_n)_n\subset {\mathcal C}_{\alpha}\cap{\mathcal C}^{\beta}$ such that
\begin{equation}\label{eq:ekeland1}
J(u_n)\leq\inf_{{\mathcal C}_{\alpha}\cap{\mathcal C}^{\beta}}J+\frac{1}{n}
\end{equation}
and
\begin{equation}\label{eq:ekeland2}
J(u_n)\leq J(u)+\frac{1}{n}\|u-u_n\|\quad\forall u\in {\mathcal C}_{\alpha}\cap{\mathcal C}^{\beta}.
\end{equation}
We choose $u=(1-t)u_n+tK(u_n),\, 0\leq t<1$. By Lemma \ref{lemma:comparison_principle} we have that $K(u_n)\subset {\mathcal C}_{\alpha}\cap{\mathcal C}^{\beta}$, hence, by convexity, $u\in {\mathcal C}_{\alpha}\cap{\mathcal C}^{\beta}$. Moreover,
\begin{equation}\label{eq:new form of u}
u=u_n-t(u_n-K(u_n)),
\end{equation}
so, when using the Taylor expansion of $J$ centered at $u_n$ we obtain that
\[
J(u)=J(u_n)-tJ'(u_n)[u_n-K(u_n)]+o(t\|u_n-K(u_n)\|).
\]
Combining this with \eqref{eq:ekeland2} and \eqref{eq:new form of u} we infer that
\[
tJ'(u_n)[u_n-K(u_n)]+o(t\|u_n-K(u_n)\|)\leq\frac{t\|u_n-K(u_n)\|}{n}.
\]
In the above inequality we apply Lemma \ref{lemma:K_pseudogradient_first_estimate} and we come to
\[\begin{split}
\frac{\|u_n-K(u_n)\|}{n}&+\frac{o(t\|u_n-K(u_n)\|)}{t}\geq\\ &\left\{\begin{array}{ll}
              C\|u_n-K(u_n)\|^2 (\|u_n\|+\|K(u_n)\|)^{p-2} \quad &\text{ if } 1<p\leq 2 \\
              C\|u_n-K(u_n)\|^p  \quad &\text{ if } p\geq 2.
                         \end{array}\right.
\end{split}\]
We fix $n$ and let $t\ri 0$ in the previous relation. Then we divide by $\|u_n-K(u_n)\|$ and we get that \begin{equation}\label{p<2}
\|u_n-K(u_n)\|\leq\frac{1}{nC} (\|u_n\|+\|K(u_n)\|)^{2-p} \quad \text{ when } 1<p\leq 2
\end{equation}
and
\begin{equation}\label{p>2}
\|u_n-K(u_n)\|^{p-1}\leq \frac{1}{nC}\quad \text{ when } p\geq 2.
\end{equation}
Now, we see from \eqref{eq:ekeland1} and Lemma \ref{lemma:boundedness} (iii) that the quantity $\|u_n\|+\|K(u_n)\|$ is bounded. Therefore, letting $n\ri\infty$ in \eqref{p<2} and \eqref{p>2} we obtain that
\[
\|u_n-K(u_n)\|\ri 0\quad\mbox{as } n\ri\infty.
\]
Next, we apply Lemma \ref{lemma:boundedness} (iv) according to the sign of $p-2$ and we have that $J'(u_n)\ri 0$ in $W^{-1,p'}(\Omega)$. Using the Palais-Smale condition proved in Lemma \ref{lemma:boundedness} (v) we deduce that $(u_n)_n$ is convergent to a $u_0$ in $W^{1,p}(\Omega)$ and this provides us the solution $u_0$ to problem \eqref{eq:main_equation}.
\end{proof}

\section{Mountain pass solution in presence of multiple sub-supersolutions}\label{section:not ordered sub-supersolutions}

In this section we prove Theorem \ref{theorem:not ordered}. We will assume that $f$ satisfies $(f1)-(f3)$ and that a strict subsolution $\alpha$ and a strict supersolution $\beta$ for \eqref{eq:main_equation} are given, with the property that ${\mathcal C}_\alpha$ and $\mathcal{C}^\beta$ are locally $K$-invariant.
\begin{lemma}\label{lemma: cones disjoint}
If $\alpha> \beta$ on a set of positive measure then there exists $\bar \epsilon$ such that $\mathcal{C}_{\alpha,\epsilon}\cap \mathcal{C}^{\beta,\epsilon}=\emptyset$ for every $0<\epsilon<\bar\epsilon$.
\end{lemma}
\begin{proof}
Assume by contradiction that for all $\epsilon>0$ there exists $u_\epsilon \in
\mathcal{C}_{\alpha,\epsilon}\cap \mathcal{C}^{\beta,\epsilon}$. Then on one hand
$\|\pi_\alpha(u_\epsilon)-\pi^\beta(u_\epsilon)\|<2\epsilon$, so that $\pi_\alpha(u_\epsilon)-\pi^\beta(u_\epsilon)\to0$ in $W_0^{1,p}(\Omega)$ as $\epsilon\to0$.
Hence, up to a subsequence, we have that, as $\epsilon\to0$,
\[
\pi_\alpha(u_\epsilon)-\pi^\beta(u_\epsilon)\to0 \text{ a.e. in } \Omega.
\]
On the other hand, $\pi_\alpha(u_\ep)\geq \alpha$ and $\pi^{\beta}(u_\ep)\leq \beta$,
thus
\[
\pi_\alpha(u_\ep)-\pi^{\beta}(u_\ep)\geq \alpha-\beta>0\quad \mbox{ on a set of positive measure, }
\]
which is a contradiction to the above relation.
\end{proof}
\begin{lemma}\label{lemma:deformation lemma}
Assume that $B\subseteq A\subseteq W_0^{1,p}(\Omega)$ are bounded closed sets with the property that $d(B,A^c)>0$ and that $inf_{u\in A}\|u-K(u)\|>\delta$ for some $\delta >0$.
Then there exist $\rho>0$ and a map $\eta:\RR^+\times W_0^{1,p}(\Omega)\ri W_0^{1,p}(\Omega)$ such that:
\begin{itemize}
\item[(i)] $\eta(0,u)=u$ for every $u\in W_0^{1,p}(\Omega)$ and $\eta(t,u)=u$ for every $u\in A^c$ and $t\in\RR^+$;
\item[(ii)] $J(\eta(\cdot,u))$ is nonincreasing for every $u\in W_0^{1,p}(\Omega)$;
\item[(iii)] $J(\eta(t,u)) \leq J(\eta(s,u))- \rho(t-s)$ if $\eta(r,u)\in B$, for every $r\in [s,t], \, 0\leq s<t$;
\item[(iv)] $\eta(t,{\mathcal C}_{\alpha,\ep}\cap A)\subset {\mathcal C}_{\alpha,\ep}$, for every $t\in \RR^+$, $0\leq\ep\leq\ep_\alpha$;
\item[(v)] $\eta(t,{\mathcal C}^{\beta,\ep}\cap A)\subset {\mathcal C}^{\beta,\ep}$, for every $t\in \RR^+$, $0\leq\ep\leq\ep_\alpha$;
\end{itemize}
where $\ep_\alpha, \ep_\beta$ are given in Definition \ref{def:locally_K_invariance}.
\end{lemma}
\begin{proof}
In the following, with an abuse of notation, we will denote by $K$ the locally Lipschitz continuous operator provided by Bartsch, Liu and Weth (see Proposition \ref{prop:lipschitz_pseudogradient}).
Take $\chi :W_0^{1,p}(\Omega)\ri [0,1]$ a regular cutoff function such that $\chi (u)=1$ if $u\in B$ and $\chi (u)=0$ if $u\in A^c$. For $u\in W_0^{1,p}(\Omega)$ we consider the following Cauchy problem:
\[
\left\{\begin{array}{ll}
\frac{d}{dt}\eta(t,u)=-\chi(\eta(t,u))\di\frac{\eta(t,u)-K(\eta(t,u))}{\|\eta(t,u)-K(\eta(t,u))\|},\\
\eta(0,u)=u.
\end{array}\right.
\]
Due to the regularity of $K$, there exists a unique solution of the above Cauchy problem, defined for every $t\in\RR^+$.

Property (i) is obvious, let us check properties (ii) and (iii). For $0\leq s<t$ we have
\begin{eqnarray*}
J(\eta(t,u))&-&J(\eta(s,u))=\int_s^t\frac{d}{dr}J(\eta(r,u))\, dr =
\int_s^tJ'(\eta(r,u))\left[\frac{d}{dr}\eta(r,u)\right]\, dr\\
&=&-\int_s^t \di\frac{\chi(\eta(r,u))}{\|\eta(r,u)-K(\eta(r,u))\|}J'(\eta(r,u))\left[\eta(r,u)-K(\eta(r,u))\right]\, dr.
\end{eqnarray*}
By Proposition \ref{prop:lipschitz_pseudogradient}, $K$ satisfies the inequalities in Lemma \ref{lemma:K_pseudogradient_first_estimate}, hence there exists $C>0$ such that
\[
J(\eta(t,u))-J(\eta(s,u))\leq - C \int_s^t \frac{\chi(\eta(r,u)) \|\eta(r,u)-K(\eta(r,u))\|}{\left(\|\eta(r,u)\|+\|K(\eta(r,u))\|\right)^{2-p}}\, dr \quad \text{for } p<2,
\]
and
\[
J(\eta(t,u))-J(\eta(s,u))\leq - C \int_s^t \chi(\eta(r,u)) \|\eta(r,u)-K(\eta(r,u))\| ^{p-1} \, dr \quad  \text{for } p\geq 2
\]
and hence (ii) is proved. Moreover, when $\eta(r,u)\in B $ for all $r\in [s,t]$, we have that $\chi(\eta(r,u))\equiv 1$ and that $\|\eta(r,u)-K(\eta(r,u))\|\geq \delta$, whereas $\|\eta(r,u)\|+\|K(\eta(r,u))\|$ is bounded, thus (iii) also holds.

Passing to the proof of (iv), we take $u\in {\mathcal C}_{\alpha,\ep}\cap A$ with $\ep\leq\ep_\alpha$ so that
\begin{equation}\label{eq:auxiliary_eq_invariance}
K({\mathcal C}_{\alpha,\ep}\cap A)\subseteq {\mathcal C}_{\alpha,\ep/2}.
\end{equation}
We have
\[
\eta(t,u)=u+t\frac{d}{dt}\eta(t,u)_{\mid_{t=0}} +o(t)=u-t\chi(u)\di\frac{u-K(u)}{\|u-K(u)\|}+o(t).
\]
We denote
$$\lambda=\di\frac{\chi(u)}{\|u-K(u)\|}$$
and we arrive at
$$\eta(t,u)=(1-t\lambda)u+t\lambda K(u)+o(t).$$
So, by \eqref{eq:auxiliary_eq_invariance} we have
\begin{eqnarray*}
\|\eta(t,u)-\pi_\alpha(\eta(t,u))\|&\leq&(1-t\lambda)\|u-\pi_\alpha(u)\|+t\lambda\|K(u)-\pi_\alpha(K(u))\|+o(t)\\
&\leq& (1-t\lambda)\ep+t\lambda\frac{\ep}{2}+o(t)=\left(1-\frac{t\lambda}{2}\right)\ep+o(t)\\
&<& \ep \quad\mbox{when }t\ri 0.
\end{eqnarray*}
Since $\eta$ has semigroup properties we deduce that $\eta(t,u)\in {\mathcal C}_{\alpha,\ep}$ for all $t\in\RR^+$.

The proof of property (v) is analogous to the proof of (iv) and it is omitted for brevity.
\end{proof}
%
%
\begin{remark}\label{remark:backward invariance}
Of course, Lemma \ref{lemma:deformation lemma} gives us a backward invariance with respect to time of the complementary of some small enlargements of the cones. More precisely, if $0<t<s$ and $\eta (s,u)\in W_0^{1,p}(\Omega)\setminus {\mathcal C}_{\alpha,\ep}$ for some $0\leq \ep\leq\ep_\alpha$, then $\eta(t,u)\in W_0^{1,p}(\Omega)\setminus {\mathcal C}_{\alpha,\ep}$. Indeed, if $\eta(t,u)\in {\mathcal C}_{\alpha,\ep}\cap A$ then $\eta(s,u)\in {\mathcal C}_{\alpha,\ep}$, whereas if $\eta(t,u)\in {\mathcal C}_{\alpha,\ep}\setminus A$ then $\eta(s,u)=u \in {\mathcal C}_{\alpha,\ep}$. Similarly, if $0<t<s$ and $\eta (s,u)\in W_0^{1,p}(\Omega)\setminus {\mathcal C}^{\beta,\ep}$ for some $0\leq \ep\leq\ep_\beta$, then $\eta(t,u)\in W_0^{1,p}(\Omega)\setminus {\mathcal C}^{\beta,\ep}$.
\end{remark}
\begin{lemma}\label{lemma: large black cones}
Given $\mathcal{U}$ bounded, let $\ep_\alpha,\ep_\beta$ be as in the Definition \ref{def:locally_K_invariance} of locally $K$-invariance, and let
$0<\ep<\min\{\ep_\alpha,\,\ep_\beta\}$.
If either $u\in ({\mathcal C}_{\alpha,\ep}\cap \mathcal{U})\setminus {\mathcal C}_{\alpha,\ep/2}$ or $u\in ({\mathcal C}^{\beta,\ep}\cap \mathcal{U})\setminus {\mathcal C}^{\beta,\ep/2}$, then $\|u-K(u)\|\geq\frac{\ep}{4}$.
\end{lemma}
\begin{proof}
Let $u\in ({\mathcal C}_{\alpha,\ep}\cap \mathcal{U})\setminus {\mathcal C}_{\alpha,\ep/2}$. Then there exists $\ep/2<r<\ep$ such that $u\in \partial {\mathcal C}_{\alpha,r}\cap \mathcal{U}$ and we deduce from the definition of locally $K$-invariance that $K(u)\in {\mathcal C}_{\alpha,r/2}$. Since
\[
\|u-\pi_\alpha(u)\|\leq \|u-\pi_\alpha(K(u))\|\leq \|u-K(u)\|+\|K(u)-\pi_\alpha(K(u))\|,
\]
then
\[
\|u-K(u)\|\geq \|u-\pi_\alpha(u)\| - \|K(u)-\pi_\alpha(K(u))\| \geq r-\frac{r}{2}.
\]
It follows that $\|u-K(u)\|\geq \frac{r}{2}\geq \frac{\ep}{4}$. Obviously, one can proceed similarly if $u\in ({\mathcal C}^{\beta,\ep}\cap \mathcal{U})\setminus {\mathcal C}^{\beta,\ep/2}$.
\end{proof}
\begin{proof}[Proof of Theorem \ref{theorem:not ordered}]
Theorem \ref{theorem:order} provides the existence of $u_1$ and $u_2$. Let us turn our attention to finding $u_3$.
We will use a mountain pass strategy. To this aim we set $\bar\epsilon$ such that
\[
{\mathcal C}_{\alpha_2,\ep}\cap {\mathcal C}^{\beta_1,\ep}=\emptyset\quad \mbox{for all }0\leq\ep\leq\bar\ep,
\]
which exists by Lemma \ref{lemma: cones disjoint}, and
\[
\Gamma=\{\gamma\in C([0,1], {\mathcal C}_{\alpha_1,\bar\ep/2}\cap {\mathcal C}^{\beta_2,\bar\ep/2}):
\quad \gamma(0)\in {\mathcal C}^{\beta_1,\bar\ep}\mbox{ and }\gamma(1)\in {\mathcal C}_{\alpha_2,\bar\ep}\},
\]
\[
S_\gamma=\{s\in [0,1]: \gamma(s)\in \overline{({\mathcal C}_{\alpha_1,\bar\ep/2}\cap {\mathcal C}^{\beta_2,\bar\ep/2})
\setminus ({\mathcal C}_{\alpha_2,\bar\ep}\cup {\mathcal C}^{\beta_1,\bar\ep} )}\},
\]
\[
c=\inf_{\gamma\in\Gamma}\max_{s\in S_\gamma}J(\gamma(s)).
\]
Note that $c\geq \inf_{{\mathcal C}_{\alpha_1,\bar\ep}\cap {\mathcal C}^{\beta_2,\bar\ep}} J>-\infty$
by Lemma \ref{lemma:boundedness} (ii). Now, let us prove that there exists a sequence $(u_n)_n\in ({\mathcal C}_{\alpha_1,\bar\ep/2}\cap {\mathcal C}^{\beta_2,\bar\ep/2})
\setminus ({\mathcal C}_{\alpha_2,\bar\ep}\cup {\mathcal C}^{\beta_1,\bar\ep} )$ such that $J(u_n)\ri c$ and $\|J'(u_n)\|_{W^{-1,p'}(\Omega)}\ri 0$. Then the result will follow by the Palais-Smale property proved in Lemma \ref{lemma:boundedness} (v).
Assume by contradiction that there exists $\delta_1>0$ such that
\begin{equation}\label{eq: contradiction mountain pass}\begin{split}
u\in ({\mathcal C}_{\alpha_1,\bar\ep/2}\cap {\mathcal C}^{\beta_2,\bar\ep/2})
&\setminus ({\mathcal C}_{\alpha_2,\bar\ep}\cup {\mathcal C}^{\beta_1,\bar\ep} ),\,\,
 c-\delta_1\leq J(u)\leq c+\delta_1\\ 
 &\mbox{ implies }\quad\|J'(u)\|_{W^{-1,p'}(\Omega)}\geq \delta_1.
\end{split}\end{equation}
By Lemma \ref{lemma:boundedness} (iv), we can find a $\delta_2>0$ such that
$\|u-K(u)\|\geq\delta_2$ for $u\in ({\mathcal C}_{\alpha_1,\bar\ep/2}\cap {\mathcal C}^{\beta_2,\bar\ep/2})
\setminus ({\mathcal C}_{\alpha_2,\bar\ep}\cup {\mathcal C}^{\beta_1,\bar\ep} )$. Then by Lemma \ref{lemma: large black cones}, there exists $\delta>0$ such that
\begin{equation*}\label{eq: }
u\in ({\mathcal C}_{\alpha_1,\bar\ep}\cap {\mathcal C}^{\beta_2,\bar\ep})
\setminus ({\mathcal C}_{\alpha_2,\bar\ep/2}\cup {\mathcal C}^{\beta_1,\bar\ep/2} ),\,\,
 c-\delta\leq J(u)\leq c+\delta\quad\mbox{ implies }\quad\|u-K(u)\|\geq \delta.
\end{equation*}
Given this $\delta$ and
\[
A=\left[(\overline{ {\mathcal C}_{\alpha_1,\bar\ep}\cap {\mathcal C}^{\beta_2,\bar\ep}})
\setminus ({\mathcal C}_{\alpha_2,\bar\ep/2}\cup {\mathcal C}^{\beta_1,\bar\ep/2} )\right] \cap\{u\in W_0^{1,p}(\Omega):\ c-\delta\leq J(u)\leq c+\delta\},
\]
\[
B=\left[(\overline{{\mathcal C}_{\alpha_1,\bar\ep/2}\cap {\mathcal C}^{\beta_2,\bar\ep/2}})
\setminus ({\mathcal C}_{\alpha_2,\bar\ep}\cup {\mathcal C}^{\beta_1,\bar\ep} )\right] \cap\left\{u\in W_0^{1,p}(\Omega):\ c-\frac{\delta}{2}\leq J(u)\leq c+\frac{\delta}{2}\right\},
\]
let $\eta$ be the deformation found in Lemma \ref{lemma:deformation lemma}. Notice that $A$ is bounded because of Lemma \ref{lemma:boundedness} (iii). Hence, by eventually choosing a smaller $\bar\ep$, we deduce from the properties (iv) and (v) of Lemma \ref{lemma:deformation lemma} that
\begin{equation}\label{eq: invariance all cones}
\eta(t,{\mathcal C}_{\alpha_i, \ep}\cap A)\subset {\mathcal C}_{\alpha_i, \ep}\quad\mbox{and}\quad\eta(t,{\mathcal C}^{\beta_i, \ep}\cap A)\subset {\mathcal C}^{\beta_i, \ep},
\end{equation}
for all $0\leq\ep\leq\bar\ep$, $t\in\RR^+$ and $i=1,2$.
Moreover, let $\rho >0$ be the quantity defined therein in property (iii). It is clear that, without loss of generality, we can choose $\rho$ such that $\rho\leq \delta$.
In the following we will denote by $\gamma $  an almost optimal path in $ \Gamma$, in the sense that
\begin{equation}\label{eq: optimal path}
\max_{s\in S_\gamma}J(\gamma(s))\leq c+\frac{\rho}{2}.
\end{equation}

Let $\bar\gamma(s)=\eta (1,\gamma(s))$.
We can see immediately that $\bar\gamma\in\Gamma$ due to relation \eqref{eq: invariance all cones} and to Lemma \ref{lemma:deformation lemma} (i). By the definition of $c$ we can find $\bar s\in (0,1)$ such that
\begin{equation}\label{eq: providing contradiction}
\bar s\in S_{\bar\gamma}\quad\mbox{and}\quad J(\bar\gamma (\bar s))\geq c.
\end{equation}
By Remark \ref{remark:backward invariance}, since $\bar s\in S_{\bar\gamma}$, then $\bar s\in S_{\gamma}$. Therefore relation \eqref{eq: optimal path}, the fact that we have chosen $\rho\leq\delta$ and the decreasing property of the flux provide
\begin{equation}\label{eq: inequalities level c}
c+\frac{\delta}{2}\geq J(\gamma(\bar s))\geq J(\eta(t, \gamma(\bar s)))\geq J(\eta(1, \gamma(\bar s)))=J(\bar\gamma(\bar s))\geq c,
\end{equation}
for every $0\leq t\leq 1$. Consequently, $\bar\gamma(\bar s)\in B$. Then, by Remark \ref{remark:backward invariance} and relation \eqref{eq: inequalities level c} we deduce that
\[
\eta (t,\gamma(\bar s))\in B \quad\mbox{for every }0\leq t\leq 1.
\]
Hence, property (iii) of Lemma \ref{lemma:deformation lemma} applies, thus providing
\[
J(\bar\gamma(\bar s))=J(\eta (1,\gamma(\bar s)))\leq J(\gamma (\bar s))-\rho.
\]
Therefore, using again relation \eqref{eq: optimal path}, we obtain
\[
J(\bar\gamma(\bar s))\leq c-\frac{\rho}{2},
\]
which contradicts relation \eqref{eq: providing contradiction}. In conclusion, we showed that \eqref{eq: contradiction mountain pass} can not hold, so that there exists a Palais-Smale sequence for $J$ at level $c$ contained in $B$. Finally, Lemma \ref{lemma:boundedness} (v) provides the existence of a solution $u_3$ to problem \eqref{eq:main_equation}. By construction $u_3\in B$, then Lemma \ref{lemma: large black cones} ensures that $u_3\in ({\mathcal C}_{\alpha_1}\cap {\mathcal C}^{\beta_2})\setminus({\mathcal C}_{\alpha_2}\cap {\mathcal C}^{\beta_1})$.
\end{proof}

\begin{remark}
Although $(f2)$ was given in its form for the homogeneity of the paper, one can see that in all the previous proofs we actually used
\begin{itemize}
\item[$(\tilde{f2})$] $f\in C(\overline{\Omega}\times\RR)$.
\end{itemize}
Additional hypotheses were included in $(f2)$ only because they are relevant for our further discussion.
\end{remark}

\section{Existence of $K$-invariant open sets}\label{section: K invariance}

Everywhere in this section we work under the hypotheses of Theorem \ref{theorem:invariance_of_the_cone}, which will be proved in several steps. In order to prove the existence of $K$-invariant open sets, it will be enough to show that $\dist(K(u),{\mathcal C}_{\alpha})=o(\dist(u,{\mathcal C}_{\alpha}))$ as $\dist(u,{\mathcal C}_{\alpha})\to0$. By property (ii) of the Lemma \ref{lemma:properties_of_distance}, we can focus on providing an upper bound to $\|[K(u)-\alpha]^-\|$. The case $p=2$ was treated in \cite{ContiMerizziTerracini1999} and we are now going to generalize it.
\begin{lemma}\label{lemma:invariance_of_cone_first_estimate}
Let $f$ satisfy $(f1)-(f3)$ and let $\alpha$ be a strict subsolution for \eqref{eq:main_equation}. Then the following holds for every $u\in W_0^{1,p}(\Omega)$
\begin{itemize}
\item[(i)] if $1<p<2$ then $\|[K(u)-\alpha]^-\|\leq C (\|K(u)\|+\|\alpha\|)^{2-p}\|[h(\cdot,\alpha+u-\pi_\alpha(u))-h(\cdot,\alpha)+a]^-\|_{L^{(p^\star)'}(\Omega)}$, where $C>0$ is a constant;
\item[(ii)] if $p\geq 2$ then $\|[K(u)-\alpha]^-\|^{p-1}\leq C \|[h(\cdot,\alpha+u-\pi_\alpha(u))-h(\cdot,\alpha)+a]^-\|_{L^{s'}(\Omega)}$, where $s=p^\star$ if $p\neq N$ and $1<s<p^\star$ if $p=N$, and $C>0$ is a constant which depends only on $s$.
\end{itemize}
\end{lemma}
\begin{proof}
Set for the moment $v=K(u)$. Obviously if $[v-\alpha]^-\equiv 0$ then there is nothing to prove, otherwise notice that
\[
\left\{\begin{array}{ll}
        -(\Delta_p v-\Delta_p\alpha)+M(|v|^{p-2}v-|\alpha|^{p-2}\alpha)=h(x,u)-h(x,\alpha)+a(x) & \quad\text{in } \Omega \\
        v-\alpha\geq 0 &\quad\text{on } \partial\Omega.
       \end{array}\right.
\]
Testing by $-[v-\alpha]^- \in W_0^{1,p}(\Omega)$, recalling that $-f\leq [f]^-$ and that $(|\xi|^{p-2}\xi-|\eta|^{p-2}\eta)\cdot(\xi-\eta)\geq 0$ for every $\eta,\xi\in\RR^N$ (by Lemma \ref{lemma:inequalities_for_vectors}), we obtain
\[
\begin{split}
\int_\Omega (|\nabla v|^{p-2}\nabla v-|\nabla\alpha|^{p-2}\nabla\alpha)&\cdot\nabla(-[v-\alpha]^-)\,dx \\ & \leq -\int_\Omega(h(x,u)-h(x,\alpha)+a(x))[v-\alpha]^-\,dx \\
&\leq \int_\Omega [h(x,u)-h(x,\alpha)+a(x)]^- [v-\alpha]^-\,dx.
\end{split}
\]
At this point notice that, by definition, $\pi_\alpha(u)\geq \alpha$ and hence $u\geq \alpha+u-\pi_\alpha(u)$ a.e. in $\Omega$. By $(f3)$ this implies $[h(x,u)-h(x,\alpha)+a(x)]^-\leq
[h(x,\alpha+u-\pi_\alpha(u))-h(x,\alpha)+a(x)]^-$  for a.e. $x\in\Omega$, so that
\begin{equation}\label{eq:auxiliary_invariance_cones}\begin{split}
\int_\Omega (|\nabla v|^{p-2}\nabla v&-|\nabla\alpha|^{p-2}\nabla\alpha)\cdot\nabla(-[v-\alpha]^-)\,dx\\& \leq
\int_\Omega [h(x,\alpha+u-\pi_\alpha(u))-h(x,\alpha)+a(x)]^- [v-\alpha]^-\,dx.
\end{split}\end{equation}
By applying first the H\"older inequality and then the Sobolev embedding, we have
\begin{equation}\nonumber\begin{split}
\int_\Omega (|\nabla v|^{p-2}\nabla v&-|\nabla\alpha|^{p-2}\nabla\alpha)\cdot\nabla(-[v-\alpha]^-)\,dx\\& \leq C \|[h(\cdot,\alpha+u-\pi_\alpha(u))-h(\cdot,\alpha)+a]^-\|_{L^{s'}(\Omega)} \|[v-\alpha]^-\|,
\end{split}\end{equation}
where $s=p^\star$ if $p\neq N$ and $1<s<p^\star$ if $p=N$.
By \eqref{eq:inequality_appearing_several_times} both cases (i) and (ii) are completed.
\end{proof}

\begin{lemma}\label{lemma:invariance_of_cone_second_estimate_general p < 2}
Let $f$ satisfy $(f1)-(f3)$ and let $\alpha$ be a strict subsolution for \eqref{eq:main_equation}. Assume either $p=2$, or $2N/(N+2)\leq p<2$ (the first inequality being strict for $N=2$) and \eqref{eq:integrability_of_1/a_p<2} holds.
Then
\[
\|[h(\cdot,\alpha+w)-h(\cdot,\alpha)+a]^-\|_{L^{s'}(\Omega)}=o(\|w\|) \quad \text{ as } \|w\|\to0,
\]
where $s=p^\star$ if $p\neq N$ and $1<s<p^\star$ if $p=N$.
\end{lemma}
\begin{proof}
Our goal is to show that from any sequence $(w_n)_n\subset W_0^{1,p}(\Omega)$ with $\|w_n\|\ri 0$ as $n\ri\infty$, we can extract a subsequence, still denoted by $(w_n)_n$, such that
\begin{equation}\label{ee9}
\lim_{n\ri\infty}\int_{\Omega}\left(\frac{[h(x,\alpha + w_n)-h(x,\alpha)+a(x)]^-}
{\|w_n\|}\right)^{s'}\,dx=0.
\end{equation}
We denote by $\Omega_0\subset\Omega$ the set
\begin{equation}\label{eq: Omega0}
\Omega_0=\{x\in\Omega:\ w_n(x)\ri 0 \mbox{ and }a(x)>0\}.
\end{equation}
Note that $|\Omega\setminus\Omega_0|=0$. Thus, if $x\in \Omega_0$, by the continuity of $h$, there exists $n_x$ such that
\[
[h(x,\alpha(x) + w_n(x))-h(x,\alpha(x))+a(x)]^-=0\quad\forall n\geq n_x.
\]
Therefore, by defining
\begin{equation}\nonumber
\varphi_n(x)=\frac{[h(x,\alpha(x) + w_n(x))-h(x,\alpha(x))+a(x)]^-}{\|w_n\|},
\end{equation}
we get that $\varphi_n\ri 0$ a.e. in $\Omega$.
In order to apply Lebesgue Theorem in relation \eqref{ee9}, we look for
\begin{equation}\label{ee10}
\varphi\in L^{s'}(\Omega) \quad \text{ such that } \quad \varphi_n(x)\leq\varphi(x) \quad \text{a.e. in } \Omega.
\end{equation}
To this aim, note first that there exists $\psi$ such that
\begin{equation}\label{eq:definition_of_psi}
\frac{|w_n|}{\|w_n\|}\leq \psi, \qquad \psi\in W_0^{1,p}(\Omega).
\end{equation}
Now, by $(f1)$ and $(f3)$ for every $k>0$ there exists $c_k>0$ such that for every $t$ with $|t|\geq k$,
\begin{equation*}
|h(x,\alpha(x)+t)-h(x,\alpha(x))|\leq\left\{\begin{array}{ll}
                          c_k|t| \quad &\text{ if } q\leq 2 \\
              			  c_k|t|^{q-1} \quad & \text{ if } q>2.
                         \end{array}\right.
\end{equation*}
Therefore for $x\in\Omega_0$ with $|w_n(x)|\geq k$ we have
\begin{equation*}
\varphi_n(x)\leq \left\{\begin{array}{ll}
                   c_k \psi \quad &\text{ if } q\leq 2  \\
             	   c_k \psi^{q-1} \quad & \text{ if } q>2,
                         \end{array}\right.
\end{equation*}
for sufficiently large $n$ since $\|w_n\|^{q-2}\to0$ as $n\to+\infty$ if $q>2$.
Given the above $k$, we consider now the situation when $|t|<k$. From $(f2)$ and $(f3)$ we deduce that there exists $\tilde{c}_k>0$ such that
\begin{equation}\label{eq:locally_holder}
|h(x,\alpha(x)+t)-h(x,\alpha(x))|\leq \tilde{c}_k|t|^{p-1} \quad \forall |t|<k.
\end{equation}
In the case $p=2$, this implies that for $x\in\Omega_0$ with $|w_n(x)|< k$ it holds $\varphi_n(x)\leq \tilde{c}_k \psi$, with $\psi$ defined in \eqref{eq:definition_of_psi}. Hence due to the previous results we have the following estimation on $\varphi_n$, for $n$ sufficiently large
\[
\varphi_n\leq\max\{c_k\psi, c_k\psi^{q-1}, \tilde{c}_k \psi\}=:\varphi.
\]
If $N\neq2$ then from the hypothesis $q<p^\star$ we infer that both $(p^\star)'$ and $(q-1)(p^\star)'$ are less than or equal to $p^\star$. Then, since $\psi\in L^{p^\star}(\Omega)$ due to the Sobolev embedding, we have that $\varphi\in L^{(p^\star)'}(\Omega)$. If $p=N=2$, then $\psi \in L^m(\Omega)$ for every $1<m<\infty$, so that \eqref{ee10} is proved also in this case.

In the case $p<2$ we need some additional work. From \eqref{eq:locally_holder} we deduce that for $x\in\Omega_0$ with $|w_n(x)|< k$ the following holds
\begin{equation}\label{eq:auxiliary1}
h(x,\alpha(x)+w_n(x))-h(x,\alpha(x))+a(x)\geq -\tilde{c}_k|w_n(x)|^{p-1}+a(x).
\end{equation}
Note that, if $|w_n(x)|< (a(x)/\tilde{c}_k)^{1/(p-1)}$, then $\varphi_n(x)\equiv0$, hence it only remains to study the situation when
\begin{equation}\label{eq:auxiliary2}
\left(\frac{a(x)}{\tilde{c}_k}\right)^{1/(p-1)}\leq |w_n(x)| \leq k.
\end{equation}
By \eqref{eq:auxiliary1} and \eqref{eq:auxiliary2} we deduce that
\[
[h(x,\alpha(x)+w_n(x))-h(x,\alpha(x))+a(x)]^- \leq \tilde{c}_k |w_n(x)|^{p-1} \leq \tilde{c}_k^{1/(p-1)} \frac{|w_n(x)|}{a(x)^{(2-p)/(p-1)}},
\]
so that, for $p<2$, we have
\begin{equation*}\label{eq: first phi_definition}
\varphi_n\leq\max\{c_k\psi, c_k\psi^{q-1}, \tilde{c}_k^{1/(p-1)} \frac{\psi}{a^{(2-p)/(p-1)}}\}=:\varphi.
\end{equation*}
The integrability of the first two terms follows without difficulties due to the hypotheses. In order to prove that the third term above belongs to $L^{(p^\star)'}(\Omega)$, it is sufficient to apply the H\"older inequality with exponents $p^\star/(p^\star)'$ and $(p^\star-1)/(p^\star-2)$, which are admissible thanks to the condition $p\geq 2N/(N+2)$, and then to use the integrability assumption \eqref{eq:integrability_of_1/a_p<2}.
\end{proof}

\begin{lemma}\label{lemma:invariance_of_cone_second_estimate_general p>2}
Let $p>2$, $f$ satisfy $(f1)-(f3)$ and $\alpha$ be a strict subsolution for \eqref{eq:main_equation}.
\begin{itemize}
\item[(i)] If $2<p<N$ and $1/a\in L^{\frac{(p-2)N}{p}}(\Omega)$, then $\|[h(\cdot,\alpha+w)-h(\cdot,\alpha)+a]^-\|_{L^{(p^\star)'}(\Omega)}=o(\|w\|^{p-1})$ as $\|w\|\to0$.
\item[(ii)] If $p=N$ and $1/a\in L^{r}(\Omega)$ for some $r>p-2$, then $\|[h(\cdot,\alpha+w)-h(\cdot,\alpha)+a]^-\|_{L^{s'}(\Omega)}=o(\|w\|^{p-1})$ as $\|w\|\to0$ for every $1<s'<r/(p-2)$.
\item[(iii)] If $p>N$ and $1/a\in L^{p-2}(\Omega)$, then $\|[h(\cdot,\alpha+w)-h(\cdot,\alpha)+a]^-\|_{L^{1}(\Omega)}=o(\|w\|^{p-1})$ as $\|w\|\to0$.
\end{itemize}
\end{lemma}
\begin{proof}
In order to treat cases (i) - (iii), we intend to show that from any sequence $(w_n)_n\subset W_0^{1,p}(\Omega)$ with $\|w_n\|\ri 0$ as $n\ri\infty$, we can extract a subsequence, still denoted by $(w_n)_n$, such that
\begin{equation}\nonumber
\lim_{n\ri\infty}\int_{\Omega}\left(\frac{[h(x,\alpha+w_n)-h(x,\alpha)+a]^-}{\|w_n\|^{p-1}}\right)^{s'}\,dx=0,
\end{equation}
with the choice of $s$ depending on the case considered. Following the argumentation from the proof of the previous lemma, we set
\begin{equation}\nonumber
\varphi_n(x)=\frac{[h(x,\alpha(x)+w_n(x))-h(x,\alpha(x))+a(x)]^-}{\|w_n\|^{p-1}}
\end{equation}
and we want to find
\begin{equation}\nonumber
\varphi\in L^{s'}(\Omega) \quad \text{ such that } \quad \varphi_n(x)\leq\varphi(x) \quad \text{a.e. in } \Omega.
\end{equation}
Due to $(f1)$ and $(f3)$, for every $k>0$ there exists $c_k>0$ such that for every $t$ with $|t|\geq k$,
\begin{equation*}
|h(x,t+\alpha(x))-h(x,\alpha(x))|\leq\left\{\begin{array}{ll}
                          c_k|t|^{p-1} \quad\text{ if } q\leq p \\
              c_k|t|^{q-1}  \quad\text{ if } q>p.
                         \end{array}\right.
\end{equation*}
We consider the set $\Omega_0$ introduced by \eqref{eq: Omega0}.
Then, for $\psi$ taken as in \eqref{eq:definition_of_psi} and $x\in\Omega_0$ with $|w_n(x)|\geq k$, we have
\begin{equation*}
\varphi_n(x)\leq \left\{\begin{array}{ll}
                          c_k \psi^{p-1} \quad\text{ if } q\leq p \\
                           \\
              c_k \psi^{q-1}  \quad\text{ if } q>p,
                         \end{array}\right.
\end{equation*}
for sufficiently large $n$ since $\|w_n\|^{q-p}\to0$ as $n\to+\infty$ if $q>p$.
For the above $k$, we discuss now the situation when $|t|<k$.
Using the fact that $h$ is locally Lipschitz, we deduce that there exists $\tilde{c}_k>0$ such that
\begin{equation*}
|h(x,\alpha(x)+t)-h(x,\alpha(x))|\leq \tilde{c}_k|t|,\qquad \forall \ |t|<k.
\end{equation*}
Hence for $x\in\Omega_0$ with $|w_n(x)|< k$ we have
\begin{equation}\label{en3}
h(x,\alpha(x)+w_n(x))-h(x,\alpha(x))+a(x)\geq -\tilde{c}_k|w_n(x)|+a(x).
\end{equation}
Note that, if $|w_n(x)|< a(x)/\tilde{c}_k$, then $\varphi_n(x)\equiv0$, hence it only remains to study the situation when
\begin{equation}\label{en4}
\frac{a(x)}{\tilde{c}_k}\leq |w_n(x)| \leq k.
\end{equation}
By \eqref{en3} and \eqref{en4} we deduce that
\[
[h(x,\alpha(x)+w_n(x))-h(x,\alpha(x))+a(x)]^- \leq \tilde{c}_k |w_n(x)| \leq \tilde{c}_k^{p-1} \frac{|w_n(x)|^{p-1}}{a(x)^{p-2}},
\]
so we have obtained the following estimation on $\varphi_n$, for $n$ sufficiently large
\begin{equation}\label{eq:phi_definition}
\varphi_n\leq\max\{c_k\psi^{q-1}, c_k\psi^{p-1}, \tilde{c}_k^{p-1}\frac{\psi^{p-1}}{a^{p-2}}\}=:\varphi,
\end{equation}
where $\psi$ is defined in \eqref{eq:definition_of_psi}. It only remains to show that $\varphi\in L^{s'}(\Omega)$, with the choice of $s$ depending on the cases (i) - (iii).

(i) Since $\psi\in L^{p^\star}(\Omega)$ we see that the first two terms in \eqref{eq:phi_definition} belong to $L^{(p^\star)'}(\Omega)$. In order to check the integrability of the third term, we apply the H\"older inequality with exponents $(p^\star -1)/(p-1)$ and $(p^\star -1)/(p^\star-p)$ as follows
\[
\int_{\Omega} \left|\frac{\psi^{p-1}}{a^{p-2}}\right|^{(p^\star)'}\,dx \leq \left(\int_\Omega |\psi|^{p^\star}\,dx\right)^{\frac{p-1}{p^\star-1}} \left(\int_\Omega \frac{1}{a^{(p-2)N/p}}\,dx\right)^{\frac{p^\star -p}{p^\star -1}},
\]
so that $\varphi\in L^{(p^\star)'}(\Omega)$ and point (ii) is proved.

(ii) In case $p=N$, the only difference with respect to the previous case is that, due to Sobolev embeddings, $\psi\in L^{m}$ for every $1<m<+\infty$. To check that $\varphi \in L^{s'}(\Omega)$ one can apply again the H\"older inequality, with exponents $r/(r-(p-2)s')$ and $r/((p-2)s')$.

(iii) If $p>N$, then $\psi\in L^\infty(\Omega)$. As a consequence, the function $\varphi$ defined in \eqref{eq:phi_definition} belongs to $L^1(\Omega)$ under the integrability assumption $1/a \in L^{p-2}(\Omega)$.
\end{proof}


The previous lemmas are providing us the tools for the proof of Theorem \ref{theorem:invariance_of_the_cone}.

\begin{proof}[Proof of Theorem \ref{theorem:invariance_of_the_cone}]
Suppose first $2N/(N+2)\leq p< 2$ (the first inequality being strict for $N=2$). We combine Lemma \ref{lemma:invariance_of_cone_first_estimate} (i) and Lemma \ref{lemma:invariance_of_cone_second_estimate_general p < 2}, with the choice $w=u-\pi_\alpha(u)$, to obtain
\[
\|[K(u)-\alpha]^-\|=C (\|K(u)\|+\|\alpha\|)^{2-p} o(\|u-\pi_\alpha(u)\|) \quad\text{ as }
\|u-\pi_\alpha(u)\|\to0.
\]
Let us consider an arbitrary bounded subset $\mathcal{U}\subset W_0^{1,p}(\Omega)$. By Remark \ref{rem:bounded norm} the set $\{K(u):\ u\in\mathcal{U}\}$ is also bounded. Recalling that $\|u-\pi_\alpha(u)\|=\dist(u,{\mathcal C}_\alpha)$ and that $\dist(K(u),{\mathcal C}_\alpha)\leq \|[K(u)-\alpha]^-\|$ (see Lemma \ref{lemma:properties_of_distance}), the previous estimates writes
\[
\dist(K(u),{\mathcal C}_{\alpha})= o(\dist(u,{\mathcal C}_{\alpha})) \quad \mbox{ for every } u\in\mathcal{U},\mbox{ as } \dist(u,{\mathcal C}_{\alpha})\ri 0.
\]
This provides the locally $K$-invariance of ${\mathcal C}_\alpha$ in the case $2N/(N+2)\leq p< 2$.

Similarly, for $p\geq2$ we prove
\[
\dist(K(u),{\mathcal C}_{\alpha})= o(\dist(u,{\mathcal C}_{\alpha})) \quad \mbox{ for every }
u\in W_0^{1,p}(\Omega),\mbox{ as } \dist(u,{\mathcal C}_{\alpha})\ri 0.
\]
Indeed, for $p=2$ we apply Lemmas \ref{lemma:invariance_of_cone_first_estimate} (ii) and \ref{lemma:invariance_of_cone_second_estimate_general p < 2}, while for $p>2$ we apply Lemmas \ref{lemma:invariance_of_cone_first_estimate} (ii) and \ref{lemma:invariance_of_cone_second_estimate_general p>2}. This provides the strict $K$-invariance in the case $p\geq2$. Since the case of a supersolution $\beta$ can be treated in a similar manner, our proof is complete.
\end{proof}

\section{A four solutions theorem}\label{section: application}
In this section we prove Theorem \ref{theorem:application} as an application of the abstract results. Let us first show the existence of strict sub-supersolutions in our context.

\begin{lemma}\label{lemma: application sub-supersolution large}
Let $p\geq 2N/(N+2)$ (the inequality being strict for $N=2$) and let $f$ satisfy $(f2)-(f5)$. Then there exist a strict subsolution $\alpha_1<0$ and a strict supersolution $\beta_2>0$ to \eqref{eq:main_equation}, with the property that ${\mathcal C}_{\alpha_1}$ and ${\mathcal C}^{\beta_2}$ are locally $K$-invariant.
\end{lemma}
\begin{proof}
By $(f2)$ and $(f4)$ there exists $g\in L^\infty(\Omega)$, $g\geq0$, such that
\begin{equation}\label{eq: application sub-supersolution large}
f(x,t)\geq \mu |t|^{p-2}t-g(x) \quad\text{ for every } t\leq0, \text{ a.e. } x\in\Omega.
\end{equation}
We consider the problem
\begin{equation}\label{eq: problem g}
\left\{\begin{array}{ll}
-\Delta_p \alpha_1 - \frac{\lambda_1+\mu}{2} |\alpha_1|^{p-2}\alpha_1 = - g(x) \quad &\text{in }\Omega \\
\alpha_1<0 & \text{on }\partial\Omega.
\end{array}\right.
\end{equation}
 Since $g\in L^\infty(\Omega)$ and $(\lambda_1+\mu)/2<\lambda_1$, the energy functional associated to the previous equation is well defined, coercive and weakly lower semicontinuous, thus it assures the existence of a solution for \eqref{eq: problem g}. Let $\alpha_1$ denote a solution of \eqref{eq: problem g}. By testing this equation with $[\alpha_1]^+\in W_0^{1,p}(\Omega)$, one sees that $\alpha_1\leq 0$ in $\Omega$, so that \eqref{eq: application sub-supersolution large} implies
\[
f(x,\alpha_1(x))\geq\mu|\alpha_1(x)|^{p-2}\alpha_1(x) -g(x).
\]
Moreover, by the strong maximum principle (\cite[Theorem 5]{Vazquez1984}, see also \cite[Theorem 2.2]{Damascelli1998}) and by our choice of the boundary conditions, $\alpha_1<0$ in $\overline{\Omega}$. We conclude that $\alpha_1$ is a strict subsolution with remainder
\[
a_1(x)\geq \frac{\mu-\lambda_1}{2} |\alpha_1(x)|^{p-2}\alpha_1(x) >0 \quad \text{ in } \overline{\Omega}.
\]
We conclude that $1/a_1 \in L^\infty(\Omega)$, so that ${\mathcal C}_{\alpha_1}$ is locally $K$-invariant by Theorem \ref{theorem:invariance_of_the_cone}. We construct $\beta_2$ in a similar way.
\end{proof}

In order to apply Theorem \ref{theorem:not ordered}, we need to find another couple of sub - supersolutions. We will find a continuum of couples of not ordered sub-supersolutions, parameterized by $l\in (0,\bar l)$, where $\bar l$ is given below. Keeping the notation $\phi_1$ for a first positive eigenfunction of $-\Delta_p$ and recalling Proposition \ref{prop: first eigenvalue and eigenfuntion}, we set
\[
\bar l=\frac{\bar t}{\|\phi_1\|_{L^\infty(\Omega)}},
\]
where $\bar t$ is such that
\begin{equation}\label{eq:f5_limit}
\frac{f(x,t)}{|t|^{p-2}t}>\frac{\lambda+\lambda_1}{2} \quad \text{ for all } |t|<\bar t, \text{ a.e. }x\in\Omega
\end{equation}
and moreover, in case $p>2$,
\begin{equation}\label{eq:f6_limit}
\left|\frac{\partial f}{\partial t}(x,t) \right|< 2(p-1) \lambda |t|^{p-2} \quad \text{ for all } |t|<\bar t,
\text{ a.e. }x\in\Omega
\end{equation}
which is possible by assumption $(f5)$. Then for every $0<l<\bar l$ we have strict subsolutions
\[
\alpha_{2,l}(x):=l\phi_1(x), \quad \text{ with remainders }\quad a_{2,l}(x)=-\lambda_1(l\phi_1(x))^{p-1}+f(x,l\phi_1(x)),
\]
and strict supersolutions
\[
\beta_{1,l}(x):=-l\phi_1(x), \quad \text{ with remainders }\quad
b_{1,l}(x)=-\lambda_1(l\phi_1(x))^{p-1}-f(x,-l\phi_1(x)).
\]
To verify the hypotheses of Theorem \ref{theorem:not ordered}, notice that, for all $l\in(0,\bar l)$, $\beta_{1,l}<\alpha_{2,l}$. Also, for sufficiently small $l$, $\alpha_1<\beta_{1,l}$ and $\alpha_{2,l}<\beta_2$. Some work is needed, but we can show that ${\mathcal C}_{\alpha_{2,l}}$ and ${\mathcal C}^{\beta_{1,l}}$ are locally $K$-invariant for every $p>({N-2+\sqrt{9N^2-4N+4}})/({2N})$. In the following we shall drop the dependence on $l$ where not explicitly needed.

 \begin{lemma}\label{lemma: co-area formula}
Assume $\varphi\in C^1(\Omega)\cap C(\partial\Omega)$ is such that $\varphi(x)>0$ for $x\in\Omega$ and $\varphi(x)=0$ for $x\in\partial\Omega$. Moreover, $\varphi$ satisfies
\[
-\nabla \varphi \cdot\nu\geq C>0\quad\mbox{on }\partial\Omega,
\]
where $C$ is a fixed constant and $\nu$ is the outer normal to $\Omega$. Then
\[
\frac{1}{\varphi^s}\in L^{1}(\Omega)\quad\mbox{for all }s\in (0,1).
\]
\end{lemma}
\begin{proof}
We start by recalling the co-area formula
\begin{equation}\label{eq: coarea}
\int_{\Omega}\frac{1}{\varphi(x)}\,dx=\int_{0}^{1}\int_{\{\omega^{-1}(\tau)\}}\frac{1}{|\nabla \omega(x)|\varphi(x)}\,dS\,d\tau,
\end{equation}
where $\omega:\Omega\ri [0,1]$ is a Lipschitz function. Since $\Omega$ is smooth, $\omega$ can be chosen such that for some positive constants $C_1,\,C_2$ and for every $x\in\Omega$,
\begin{equation}\label{eq: omega prop 1}
\omega(x)\leq C_1 \dist(x,\partial \Omega),
\end{equation}
\begin{equation}\label{eq: omega prop 2}
\frac{1}{|\nabla \omega(x)|}\leq C_2.
\end{equation}
Given any $x\in \Omega$, let $y(x)$ be the point belonging to $\partial\Omega$ which satisfies $|y(x)-x|=\dist(x,\partial\Omega)$. Then we have
\[
\begin{split}
\varphi(x)&=\varphi(x)-\varphi(y(x))=\int_0^1 \frac{d}{dt} \varphi(tx+(1-t)y(x))dt=\\
&=\int_0^1 \nabla \varphi(tx+(1-t)y(x))\cdot(x-y(x)) dt=\\
&=\int_0^1 -\nabla \varphi(tx+(1-t)y(x))\cdot \frac{y(x)-x}{|y(x)-x|} \cdot |y(x)-x| dt=\\
&\geq C \int_0^1 |y(x)-x| dt,
\end{split}
\]
that is,
\begin{equation}\label{eq about distance}
\varphi(x)\geq C\dist(x,\partial\Omega).
\end{equation}
This, together with \eqref{eq: omega prop 1}, implies
\[
 \varphi(x)\geq \frac{C}{C_1}\tau\quad \mbox{for every }x\in \{\omega^{-1}(\tau)\}.
\]
By the above relation, the co-area formula \eqref{eq: coarea} and the inequality \eqref{eq: omega prop 2}, we get
\[
\int_{\Omega}\frac{1}{\varphi^s(x)}\, dx \leq C_2 \left(\frac{C_1}{C}\right)^s \int_{0}^{1}|\{\omega^{-1}(\tau)\}| \frac{1}{\tau^s}\,d\tau,
\]
which is integrable for $s\in(0,1)$.
\end{proof}

\begin{lemma}\label{lemma: application K invariance p<2}
Let $({N-2+\sqrt{9N^2-4N+4}})/({2N})<p<2$ and let $f$ satisfy $(f2)-(f5)$. Then ${\mathcal C}_{\alpha_2}$ and ${\mathcal C}^{\beta_1}$ are locally $K$-invariant.
\end{lemma}
\begin{proof}
Notice that \eqref{eq:f5_limit} holds with $t=\alpha_{2}$ and $t=\beta_{1}$, so that
\[
a_{2}(x),b_{1}(x)>\frac{\lambda-\lambda_1}{2} (l\phi_1(x))^{p-1}\quad \mbox{for a.e. } x\in\Omega.
\]
Since $\phi_1$ satisfies the Hopf lemma (see Proposition \ref{prop: first eigenvalue and eigenfuntion}), we deduce from Lemma \ref{lemma: co-area formula} that
\[
\left(\frac{1}{a_2}\right)^{s/(p-1)}, \left(\frac{1}{b_1}\right)^{s/(p-1)} \in L^1(\Omega) \quad\mbox{for all }s\in (0,1).
\]
At this point it is not difficult to check that $1/a_2$ and $1/b_1$ satisfy the integrability condition \eqref{eq:integrability_of_1/a_p<2} whenever $({N-2+\sqrt{9N^2-4N+4}})/({2N})<p<2$, so that the locally $K$-invariance is a direct consequence of Theorem \ref{theorem:invariance_of_the_cone} (i).
\end{proof}

In order to prove the $K$-invariance for every $p\geq2$ we need one more estimate.
\begin{lemma}\label{lemma: application estimates}
Let $p\geq 2$ and let $f$ satisfy $(f2)-(f5)$, then
\[
\|[h(\cdot,\alpha_2+w)-h(\cdot,\alpha_2)+a_2]^-\|_{L^{s'}(\Omega)}=o(\|w\|^{p-1}) \quad \text{ as } \|w\|\to0,
\]
where $s=p^\star$ if $p\neq N$ and $1<s<p^\star$ if $p=N$. An analogous estimate holds for $\beta_1$.
\end{lemma}
\begin{proof}
We proceed as in the proof of Lemma \ref{lemma:invariance_of_cone_second_estimate_general p>2}. For every $0<l<\bar{l}$ we fix
\[
0<k<\frac{\bar{l}}{l}-1.
\]
If $|w(x)|\geq k \alpha_2(x)$ then assumptions $(f4)$ and $(f5)$ imply
\[\begin{split}
|h(x,\alpha_2+w)-h(x,\alpha_2)|&\leq |h(x,\alpha_2+w)|+|h(x,\alpha_2)|\\&\leq C(|\alpha_2+w|^{p-1}+|\alpha_2|^{p-1})
\leq c_k |w|^{p-1}.
\end{split}\]
We discuss now the case $|w(x)|<k \alpha_2(x)$ with $x\in\Omega_0$ ($\Omega_0$ was introduced in \eqref{eq: Omega0}). By Taylor expansion with Lagrange remainder there exists $\xi(x)\in(\alpha_2(x)-w(x),\alpha_2(x)+w(x))$ such that
\[
h(x,\alpha_2(x)+w(x))-h(x,\alpha_2(x))=\left(\frac{\partial f}{\partial t}(x,\xi(x))+M(p-1)|\xi(x)|^{p-2}\right)w(x).
\]
Since $|\xi(x)|\leq (k+1) \alpha_2(x)\leq \bar l \|\phi_1\|_{L^\infty(\Omega)}$, then \eqref{eq:f6_limit} holds with $t=\alpha_2(x)$. We deduce the existence of $\tilde{d}_k>0$ such that
\begin{equation}\label{eq:taylor_expansion}
|h(x,\alpha_2+w)-h(x,\alpha_2)|\leq \tilde{d}_k \alpha_2^{p-2} w.
\end{equation}
Hence for $x\in\Omega_0$ with $|w(x)|<k \alpha_2(x)$ we have
\[
\begin{split}
h(x,\alpha_2+w)-h(x,\alpha_2)+a_2(x) & \geq -\tilde{d}_k \alpha_2^{p-2}|w|+f(x,\alpha_2)-\lambda_1 \alpha_2^{p-1}\\
&\geq -\tilde{d}_k \alpha_2^{p-2}|w| +\frac{\lambda-\lambda_1}{2} \alpha_2^{p-1},
\end{split}
\]
where we used \eqref{eq:f5_limit} in the last inequality. Now, if $|w|<(\lambda-\lambda_1)\alpha_2/(2\tilde{d}_k)$ then
$[h(x,\alpha_2(x)+w(x))-h(x,\alpha_2(x))+a(x)]^-\equiv0$, whereas in the complementary case we deduce from \eqref{eq:taylor_expansion} that
\[
\begin{split}
[h(x,\alpha_2(x)+w(x))-h(x,\alpha_2(x))+a(x)]^- & \leq \tilde{d}_k \alpha_2^{p-2} |w|^{p-1} \frac{1}{|w|^{p-2}} \\
& \leq \tilde{d}_k \alpha_2^{p-2} |w|^{p-1} \left(\frac{2\tilde{d}_k}{\lambda-\lambda_1}\right)^{p-2} \frac{1}{\alpha_2^{p-2}}.
\end{split}
\]
Therefore there exists $\tilde{c}_k>0$ such that
\[
[h(x,\alpha_2(x)+w(x))-h(x,\alpha_2(x))+a(x)]^- \leq \tilde{c}_k |w|^{p-1} \quad \text{ if } |w(x)|\leq k\alpha_2(x), x\in \Omega_0
\]
and the conclusion follows as in Lemma \ref{lemma:invariance_of_cone_second_estimate_general p>2}.
\end{proof}
%
The result above allows to prove the $K$-invariance for every $p\geq2$.
\begin{lemma}\label{lemma: application K invariance p>2}
Let $p\geq 2$ and let $f$ satisfy $(f2)-(f5)$, then ${\mathcal C}_{\alpha_2}$ and ${\mathcal C}^{\beta_1}$ are strictly $K$-invariant.
\end{lemma}
For brevity we omit the proof of this lemma since it is an easy adaptation of the arguments from Section \ref{section: K invariance}, with the aid of Lemma \ref{lemma: application estimates}.

In what follows, we prove another auxiliary result that is needed in the argumentation of Theorem \ref{theorem:application}.

\begin{lemma}\label{lemma:compactness_brezis_kato}
Let $f$ satisfy $(f2)-(f5)$ and let $(u_n)_n$ be a sequence of solutions to \eqref{eq:main_equation}, bounded in $W_0^{1,p}(\Omega)$. Then there exists $\bar u \in W_0^{1,p}(\Omega)$ and $\gamma\in (0,1)$ such that $u_n\to \bar u$ in $C^{1,\gamma}(\Omega)$ and $-\Delta_p \bar u=f(x,\bar u)$.
\end{lemma}
\begin{proof}
One can prove, via a Brezis-Kato argument, that for every $1<s<+\infty$ there exists a constant $C(s)$, depending only on $s$, such that $\|u_n\|_{L^s(\Omega)} \leq C(s)$ for every $n$. Given this, a standard regularity result (see for example \cite[Appendix Theorem E.0.19]{Peral1997}) provides $\|u_n\|_{L^\infty(\Omega)} \leq C$ for every $n$. The regularity theory in \cite{DiBenedetto1983,Lieberman1988,Tolksdorf1984} then provides the existence of $\gamma'\in(0,1)$ such that $\|u_n\|_{C^{1,\gamma'}(\Omega)} \leq C$ for every $n$. The result then follows from the compactness of the immersion $C^{1,\gamma}(\Omega)\hookrightarrow C^{1,\gamma'}(\Omega)$ for every $0<\gamma<\gamma'$.
\end{proof}

\begin{proof}[Proof of Theorem \ref{theorem:application}]
We can suppose that
\begin{equation}\label{eq:application_sign_of_f}
f(x,t)t\geq 0 \quad \text{ for every } t, \text{ a.e. } x\in\Omega,
\end{equation}
otherwise the conclusion follows as a consequence of the study conducted in \cite{BartschLiuWeth2005} (see assumption (H$_3'$) therein).
We proved in the previous lemmas that Theorem \ref{theorem:not ordered} applies for every $0<l<\bar l$, thus providing
a negative solution $u_{1,l}$, a positive solution $u_{2,l}$ and a third solution $u_{3,l}$ satisfying
\[
u_{1,l}\in {\mathcal C}_{\alpha_1}\cap {\mathcal C}^{\beta_{1,l}}, \quad
u_{2,l} \in {\mathcal C}_{\alpha_{2,l}}\cap {\mathcal C}^{\beta_2}, \quad
u_3\in ({\mathcal C}_{\alpha_1}\cap {\mathcal C}^{\beta_2})\setminus ({\mathcal C}_{\alpha_{2,l}}\cup {\mathcal C}^{\beta_{1,l}}).
\]
It only remains to show that $u_{3,l}$ changes sign for sufficiently small $l$.
Let us first prove that $u_{3,l}\not\equiv 0$. Thanks to Proposition \ref{prop: second eigenvalue} and to the fact that $C^1_0(\overline{\Omega})$ is dense in $W_0^{1,p}(\Omega)$, there exists $\gamma\in C([0,1],C^1_0(\overline{\Omega}))$ such that
\[
\gamma(0)=\beta_{1,l}, \ \gamma(1)=\alpha_{2,l}, \quad \int_\Omega|\gamma(s)|^p\,dx=l^p, \quad
\max_{s\in[0,1]}\int_\Omega |\nabla\gamma(s)|^p\,dx\leq \lambda_2 l^p .
\]
By choosing $l$ sufficiently small we have $\gamma(s) \in {\mathcal C}_{\alpha_1,\bar\ep/2}\cap {\mathcal C}^{\beta_2,\bar\ep/2}$ and, by $(f5)$, $F(\gamma(s))>\frac{\lambda_2 |\gamma(s)|^p}{p}$, for every $s\in [0,1]$. Hence $J(\gamma(s))<0$ for every $s\in [0,1]$. Because of the mountain pass characterization of $u_{3,l}$, we deduce that $u_{3,l}\not\equiv0$.

Let us assume by contradiction that $u_{3,l}\geq0$ for every $l>0$. To simplify the notation we denote by
\[
u_n:=u_{3,l_n}
\]
a sequence of solutions with $l_n\to0$ as $n\to+\infty$. Since $u_n\in {\mathcal C}_{\alpha_1}\cap{\mathcal C}^{\beta_2}$, Lemma \ref{lemma:compactness_brezis_kato} implies the existence of $\bar u\in W_0^{1,p}(\Omega)$ and $\gamma>0$ such that
\begin{equation}\label{eq:application_convergence_c_1}
u_n\to \bar u \quad \text{ in }C^{1,\gamma}(\Omega) .
\end{equation}
Obviously, $\bar u\geq 0$. Let us first remark that $\bar u\not\equiv0$. Indeed, we proved above that $J(u_n)<0$ for every $n$. Moreover, the variational characterization of $u_n$ implies that $J(u_n)\leq J(u_m)$ whenever $m<n$, since the min-max level is nonincreasing as $l_n\to0$. Therefore $J(\bar u)=\lim_{n\to+\infty} J(u_n)<0$, so that $\bar u\not\equiv0$. Since by Lemma \ref{lemma:compactness_brezis_kato} we know that $\bar u$ solves $-\Delta_p \bar u=f(x,\bar u)$ and by \eqref{eq:application_sign_of_f} we have $f(x, \bar u)\geq0$, the strong maximum principle and the generalized Hopf lemma (see \cite[Theorem 5]{Vazquez1984}) imply
\begin{equation}\label{eq:application_contradiction_sign_changing}
\bar u>0 \quad\text{ in }\Omega \quad\text{ and }\quad \partial_\nu\bar u<0 \quad \text{ on }\partial \Omega.
\end{equation}
Let us consider the set
\[
\Omega_n^-=\{x\in\Omega: \ u_n(x)<l_n\phi_1(x)\}.
\]
The localization of $u_n$ implies $|\Omega_n^-|>0$ for every $n$. We fix $\bar x\in \Omega_n^-$ and we let $y(\bar x)$ be the point belonging to $\partial\Omega$ which satisfies $|y(\bar x)-\bar x|=\dist(\bar x,\partial\Omega)$. On one hand, since $\phi_1\in C^1(\overline\Omega)$, there exists $C>0$ such that
\begin{equation}\label{eq at the end}
u_n(\bar x)< l_n\phi_1(\bar x)\leq C l_n \dist(\bar x,\partial\Omega).
\end{equation}
On the other hand, $\bar u$ satisfies \eqref{eq:application_contradiction_sign_changing} and proceeding as in the proof of Lemma \ref{lemma: co-area formula} we have
$$
\bar u (\bar x)\geq C\dist(\bar x,\partial\Omega).
$$
By the above relation and \eqref{eq:application_convergence_c_1}, there exists $C>0$ such that
\begin{equation}\label{eq about distance}
u_n(\bar x)\geq C \dist(\bar x,\partial\Omega).
\end{equation}

From \eqref{eq at the end} and \eqref{eq about distance} we infer the existence of a positive constant $C$ such that $l_n\geq C$, which is a contradiction for $n$ large.
\end{proof}

\section{Additional $K$-invariance results}\label{section: comments}

For the clarity of our work, we avoided possible ramifications of the discussion, but we can not ignore the fact that such ramifications exist. For example, there are alternatives to the conditions that ensured the $K$-invariance of open sets (see Theorem \ref{theorem:invariance_of_the_cone}) and we are going to present them.

\begin{theorem}\label{newtheorem:invariance_of_the_cone}
Let $f$ satisfy $(f1)-(f3)$ and $q$ be the exponent from the growth condition $(f1)$. Let $\alpha$ be a strict subsolution and $\beta$ be a strict supersolution for \eqref{eq:main_equation}, with remainders $a,b$ respectively, given in Definition \ref{definition_subsolution_supersolution}. Then
\begin{itemize}
\item[(i)]  ${\mathcal C}_{\alpha}$ and ${\mathcal C}^{\beta}$ are locally $K$-invariant if $2N/(N+1)\leq p< 2$ (the first inequality being strict for $N=2$), $ q\leq p^\star-p/(N-p)$  and
\begin{equation}\label{eq:integrability_of_1/a_p<2 Hardy}
\frac{\dist(\cdot,\partial\Omega)}{a^{({2-p})/({p-1})}},\,\frac{\dist(\cdot,\partial\Omega)}{b^{({2-p})/({p-1})}}\in L^{\infty}(\Omega);
\end{equation}
\item[(ii)]  ${\mathcal C}_{\alpha}$ and ${\mathcal C}^{\beta}$ are strictly $K$-invariant if $p>2$,
\begin{equation}\label{eq:first integrability_of_1/a for Hardy}
\frac{\dist(\cdot,\partial\Omega)^p}{a^{p-2}},\,\frac{\dist(\cdot,\partial\Omega)^p}{b^{p-2}}\in L^{\infty}(\Omega)
\end{equation}
and, in addition, either $p\geq N>2$ or $q\leq p^\star -p/(N-p)$;
\item[(iii)] ${\mathcal C}_{\alpha}$ and ${\mathcal C}^{\beta}$ are strictly $K$-invariant if $2<p< N$ with $p^\star -p/(N-p)< q<p^\star$ and
\begin{equation}\label{eq:second-first integrability_of_1/a for Hardy}
\frac{\dist(\cdot,\partial\Omega)^{p+N+q-1-\frac{Nq}{p}}}{a^{p-2}},\,
\frac{\dist(\cdot,\partial\Omega)^{p+N+q-1-\frac{Nq}{p}}}{b^{p-2}}\in L^{\infty}(\Omega).
\end{equation}
\end{itemize}
\end{theorem}

\begin{remark}
 When $2<p< N $ and $ p^\star -p/(N-p)\leq q<p^\star$, assumption $\frac{\dist(\cdot,\partial\Omega)^{p+N+q-1-\frac{Nq}{p}}}{a^{p-2}}\in L^{\infty}(\Omega)$ implies assumption $\frac{\dist(\cdot,\partial\Omega)^p}{a^{p-2}}\in L^{\infty}(\Omega)$, thus it is more restrictive. However, it is not too restrictive, since $p+N+q-1-\frac{Nq}{p}>0$.
\end{remark}

For the proof of Theorem \ref{newtheorem:invariance_of_the_cone} we first recall that from the usual Hardy inequality $\|u/\dist(\cdot,\partial\Omega)\|_{L^p(\Omega)}\leq C \|u\|$,  a generalized inequality can be recovered.

\begin{lemma}\label{lemma: generalized Hardy}
There exists $C>0$ such that for every $u \in W_0^{1,p}(\Omega)$ it holds
$$\left\|\frac{u}{\dist(\cdot,\partial\Omega)^t}\right\|_{L^s(\Omega)}\leq C \|u\|,$$
provided that $s\geq 1$, $0<t< 1$, $ts< p$ and, if $p< N$, $(s-ts)/(p-ts)\leq p^\star/p$.
\end{lemma}
\begin{proof}
We apply the H\"older inequality, with exponents $p/(ts)$ and $p/(p-ts)$ respectively, as follows
\[
\left\|\frac{u}{\dist(\cdot,\partial\Omega)^t}\right\|_{L^s(\Omega)} \leq
\left\|\frac{u}{\dist(\cdot,\partial\Omega)}\right\|^t_{L^p(\Omega)}\left(\int_\Omega u^{sp(1-t)/(p-ts)} \,dx\right)^{(p-ts)/(sp)},
\]
which is allowed by the assumptions on $t$ and $s$. Now, the first term in the right hand side is controlled by $C \|u\|^t$ thanks to the usual Hardy inequality. As for the second term, the assumptions on $s$ and $t$ ensure it is bounded by $C \|u\|^{1-t}$ by the continuous Sobolev embedding.
\end{proof}

In the same manner as in Section \ref{section: K invariance}, we rely on auxiliary lemmas to carry on our work.

\begin{lemma}\label{lemma: additional 1}
Let $f$ satisfy $(f1)-(f3)$ and let $\alpha$ be a strict subsolution for \eqref{eq:main_equation}. Then there exists a constant $C>0$ such that the following hold for every $u\in W_0^{1,p}(\Omega)$:
\item[(i)] \[\begin{split}\|[h(\cdot,\alpha+u&-\pi_\alpha(u))-h(\cdot,\alpha)+a]^- \dist(\cdot,\partial\Omega)\|_{L^{p'}(\Omega)}\\
     &\geq\left\{\begin{array}{ll}
C(\|K(u)\|+\| \alpha \|)^{p-2}\|[K(u)-\alpha]^-\| & \text{ if } p< 2\\
C\|[K(u)-\alpha]^-\|^{p-1}  & \text{ if }p> 2;
\end{array}\right.\end{split}\]
\item[(ii)] if $p<N$ and $p^\star-p/(N-p)< q<p^\star$ then $\|[K(u)-\alpha]^-\|^{p-1}\leq C \|[h(\cdot,\alpha+u-\pi_\alpha(u))-h(\cdot,\alpha)+a]^- \dist(\cdot,\partial\Omega)^{p(p^\star-q)/(p^\star-p)}\|_{L^{p^\star/(q-1)}(\Omega)}$.
\end{lemma}
\begin{proof}
Set for the moment $v=K(u)$.
As in the proof of Lemma \ref{lemma:invariance_of_cone_first_estimate}, we have that \eqref{eq:auxiliary_invariance_cones} holds. We apply inequality  \eqref{eq:inequality_appearing_several_times} and we obtain
\begin{equation}\label{eq:auxiliary_invariance_cones2}\begin{split}
\int_\Omega [h(x,\alpha+u-\pi_\alpha(u))&-h(x,\alpha)+a(x)]^- [v-\alpha]^-\,dx\\&\geq
\left\{\begin{array}{ll}
C(\|v\|+\| \alpha \|)^{p-2}\|[v-\alpha]^-\|^2  & \text{ if } p< 2\\
C \|[v-\alpha]^-\|^p  & \text{ if }p> 2.
\end{array}\right.\end{split}
\end{equation}
To prove (i), we multiply and divide the left hand side in \eqref{eq:auxiliary_invariance_cones2} by $\dist(x,\partial\Omega)$. Then by the H\"older inequality and by the standard Hardy inequality, we get
\[\begin{split}
\|[h(\cdot,\alpha+u-\pi_\alpha(u))&-h(\cdot,\alpha)+a]^- \dist(\cdot,\partial\Omega)\|_{L^{p'}(\Omega)}\\&\geq\left\{\begin{array}{ll}
C(\|v\|+\| \alpha \|)^{p-2}\|[v-\alpha]^-\| & \text{ if } p< 2\\
C\|[v-\alpha]^-\|^{p-1}  & \text{ if }p> 2.
\end{array}\right.\end{split}
\]

As for (ii), we multiply and divide the right hand side in \eqref{eq:auxiliary_invariance_cones2} by $\dist(x,\partial\Omega)^{p(p^\star-q)/(p^\star-p)}$ and then we apply the H\"older inequality with exponents $p^\star/(q-1)$ and $p^\star/(p^\star-q+1)$. We obtain
\[\begin{split}
\|[v-\alpha]^-\|^{p}\leq C
\|[h(\cdot,\alpha+u-\pi_\alpha(u))&-h(\cdot,\alpha)+a]^-\dist(\cdot,\partial\Omega)^{\frac{p(p^\star-q)}{p^\star-p}}
\|_{L^{\frac{p^\star}{q-1}}(\Omega)}\\ &\cdot
\left\|\frac{[v-\alpha]^-}{\dist(\cdot,\partial\Omega)^{\frac{p(p^\star-q)}{p^\star-p}}}
\right\|_{L^{\frac{p^\star}{p^\star-q+1}}(\Omega)}.
\end{split}\]
One can check that the assumptions of the generalized Hardy inequality from Lemma \ref{lemma: generalized Hardy} are satisfied under our hypotheses, so that
\[
\left\|\frac{[v-\alpha]^-}{\dist(\cdot,\partial\Omega)^{p(p^\star-q)/(p^\star-p)}}\right\|_{L^{p^\star/(p^\star-q+1)}(\Omega)} \leq
C \|[v-\alpha]^-\|
\]
and this completes the proof.
\end{proof}


\begin{lemma}\label{lemma:invariance_of_cone_third_estimate}
Let $f$ satisfy $(f1)-(f3)$ and let $\alpha$ be a strict subsolution for \eqref{eq:main_equation}.
\begin{itemize}
\item[(i)] Assume that $2N/(N+1)\leq p< 2$ (the first inequality being strict for $N=2$), $q\leq p^\star-p/(N-p)$  and
$a$ fulfills property \eqref{eq:integrability_of_1/a_p<2 Hardy}. Then $\|[h(\cdot,\alpha+w)-h(\cdot,\alpha)+a]^-\dist(\cdot,\partial\Omega)\|_{L^{p'}(\Omega)}=o(\|w\|)$ as $\|w\|\to0$.
\item[(ii)] Assume that $p>2$ and $a$ fulfills property \eqref{eq:first integrability_of_1/a for Hardy}. If either $p\geq N>2$, or $q\leq p^\star -p/(N-p)$, then $\|[h(\cdot,\alpha+w)-h(\cdot,\alpha)+a]^-\dist(\cdot,\partial\Omega)\|_{L^{p'}(\Omega)}=o(\|w\|^{p-1})$ as $\|w\|\to0$.
\item[(iii)] Assume that $2<p< N$ with $p^\star -p/(N-p)< q<p^\star$ and $a$ fulfills property \eqref{eq:second-first integrability_of_1/a for Hardy}. Then $\|[h(\cdot,\alpha+w)-h(\cdot,\alpha)+a]^-
    \dist(\cdot,\partial\Omega)^{\frac{p(p^\star-q)}{p^\star-p}}\|_{L^{\frac{p^\star}{q-1}}(\Omega)}=o(\|w\|^{p-1})$ as $\|w\|\to0$.
\end{itemize}
\end{lemma}
\begin{proof}
(i) The idea is to prove that from any sequence $(w_n)_n\subset W_0^{1,p}(\Omega)$ with $\|w_n\|\ri 0$ as $n\ri\infty$, we can extract a subsequence, still denoted by $(w_n)_n$, such that
\begin{equation}\nonumber
\lim_{n\ri\infty}\int_{\Omega}\left(\frac{[h(x,w_n+\alpha)-h(x,\alpha)+a]^-\dist(x,\partial\Omega)}
{\|w_n\|}\right)^{p'}\,dx=0.
\end{equation}
We set
 \begin{equation}\nonumber
\varphi_n(x)=\frac{[h(x,w_n(x)+\alpha(x))-h(x,\alpha(x))+a(x)]^-\dist(x,\partial\Omega)}{\|w_n\|}.
\end{equation}
We intend to find
\begin{equation}\nonumber
\varphi\in L^{p'}(\Omega) \quad \text{ such that } \quad \varphi_n(x)\leq\varphi(x) \quad \text{a.e. in } \Omega.
\end{equation}
The argumentation follows as in the proof of Lemma \ref{lemma:invariance_of_cone_second_estimate_general p < 2}, the case $p < 2$. Using $\psi$ provided by \eqref{eq:definition_of_psi}, we arrive at the following estimation on $\varphi_n$, for $n$ sufficiently large:
\begin{equation}\nonumber
\varphi_n\leq\max\{c_k\psi\dist(\cdot,\partial\Omega), c_k\psi^{q-1}\dist(\cdot,\partial\Omega), \tilde{c}_k^{1/(p-1)}\frac{\psi}{a^{(2-p)/(p-1)}}\dist(\cdot,\partial\Omega)\}=:\varphi.
\end{equation}
We verify the integrability conditions. We have
$$
\int_\Omega \left(|\psi|\dist(x,\partial\Omega)\right)^{p'}\,dx\leq C \int_\Omega |\psi|^{p'}\,dx,
$$
  \begin{equation}\label{eq: integrability for 1 term}\int_\Omega \left(|\psi|^{q-1}\dist(x,\partial\Omega)\right)^{p'}\,dx\leq C \int_\Omega |\psi|^{(q-1)p/(p-1)}\,dx,
  \end{equation}
  $$\int_\Omega \left(\frac{|\psi|}{a^{(2-p)/(p-1)}}\dist(x,\partial\Omega)\right)^{p'}\,dx
\leq \left\|\frac{\dist(\cdot,\partial\Omega)}{a^{(2-p)/(p-1)}}\right\|^{p'}_{L^\infty(\Omega)}
\int_\Omega{|\psi|}^{p'}\,dx.$$
 All these quantities are finite because $\psi\in W_0^{1,p}$ and we know that $p\geq 2N/(N+1)$, $ q\leq p^\star-p/(N-p)$ and $a$ satisfies \eqref{eq:integrability_of_1/a_p<2 Hardy}.

\smallskip

(ii) We show that from any sequence $(w_n)_n\subset W_0^{1,p}(\Omega)$ with $\|w_n\|\ri 0$ as $n\ri\infty$, we can extract a subsequence, still denoted by $(w_n)_n$, such that
\begin{equation}\nonumber
\lim_{n\ri\infty}\int_{\Omega}\left(\frac{[h(x,w_n+\alpha)-h(x,\alpha)+a]^-\dist(x,\partial\Omega)}
{\|w_n\|^{p-1}}\right)^{p'}\,dx=0.
\end{equation}
Taking $\varphi_n$ of the form
 \begin{equation}\nonumber
\varphi_n(x)=\frac{[h(x,w_n(x)+\alpha(x))-h(x,\alpha(x))+a(x)]^-\dist(x,\partial\Omega)}{\|w_n\|^{p-1}},
\end{equation}
we search for
\begin{equation}\nonumber
\varphi\in L^{p'}(\Omega) \quad \text{ such that } \quad \varphi_n(x)\leq\varphi(x) \quad \text{a.e. in } \Omega,
\end{equation}
and we follow the same steps as in the proof of Lemma \ref{lemma:invariance_of_cone_second_estimate_general p>2}, with $\psi$ given by \eqref{eq:definition_of_psi}.
Thus, for $n$ sufficiently large, we  obtain the estimation:
\begin{equation}\nonumber
\varphi_n\leq\max\{c_k\psi^{q-1}\dist(\cdot,\partial\Omega), c_k\psi^{p-1}\dist(\cdot,\partial\Omega), \tilde{c}_k^{p-1}\frac{\psi^{p-1}}{a^{p-2}}\dist(\cdot,\partial\Omega)\}=:\varphi.
\end{equation}
When checking the integrability condition for the first term, we get again inequality \eqref{eq: integrability for 1 term}
    which is convenient since $\psi\in W_0^{1,p}$ and either $p\geq N$ or $q\leq p^\star -p/(N-p)$. The integrability of the second term is trivial.
  As for the third term,  we use condition \eqref{eq:first integrability_of_1/a for Hardy} to get
$$\int_\Omega \left(\frac{|\psi|^{p-1}}{a^{p-2}}\dist(x,\partial\Omega)\right)^{p'}\,dx
\leq \left\|\frac{\dist(\cdot,\partial\Omega)^p}{a^{p-2}}\right\|^{p'}_{L^\infty(\Omega)}
\int_\Omega\left(\frac{|\psi|}{\dist(x,\partial\Omega)}\right)^p\,dx,$$
then we conclude that $\varphi\in L^{p'}(\Omega)$ by using the Hardy inequality.

\smallskip

(iii) Now we prove that from any sequence $(w_n)_n\subset W_0^{1,p}(\Omega)$ with $\|w_n\|\ri 0$ as $n\ri\infty$, we can extract a subsequence, still denoted by $(w_n)_n$, such that
\begin{equation}\nonumber
\lim_{n\ri\infty}\int_{\Omega}\left(\frac{[h(x,w_n+\alpha)-h(x,\alpha)+a]^-\dist(x,\partial\Omega)^{\frac{p(p^\star-q)}{p^\star-p}}}
{\|w_n\|^{p-1}}\right)^{\frac{p^\star}{q-1}}\,dx=0.
\end{equation}
This time we take $\varphi_n$ of the form
 \begin{equation}\nonumber
\varphi_n(x)=\frac{[h(x,w_n(x)+\alpha(x))-
h(x,\alpha(x))+a(x)]^-\dist(x,\partial\Omega)^{\frac{p(p^\star-q)}{p^\star-p}}}{\|w_n\|^{p-1}}.
\end{equation}
Then we try to find
\begin{equation}\nonumber
\varphi\in L^{\frac{p^\star}{q-1}}(\Omega) \quad \text{ such that } \quad \varphi_n(x)\leq\varphi(x) \quad \text{a.e. in } \Omega.
\end{equation}
We repeat the previous arguments and,
 for $n$ sufficiently large and $\psi$ given by \eqref{eq:definition_of_psi}, we  obtain an estimation on $\varphi_n$ provided by the following choice of $\varphi$:
\begin{equation}\nonumber
\max\{c_k\psi^{q-1}\dist(\cdot,\partial\Omega)^{\frac{p(p^\star-q)}{p^\star-p}}, c_k\psi^{p-1}\dist(\cdot,\partial\Omega)^{\frac{p(p^\star-q)}{p^\star-p}}, \tilde{c}_k^{p-1}\frac{\psi^{p-1}}{a^{p-2}}\dist(\cdot,\partial\Omega)^{\frac{p(p^\star-q)}{p^\star-p}}\}.
\end{equation}

It is not difficult to verify the integrability of the first two terms because $\psi\in L^{p^\star}(\Omega)$ and $p < p^\star -p/(N-p)<q$. For the third term we apply hypothesis \eqref{eq:second-first integrability_of_1/a for Hardy} to obtain
\[\begin{split}&\int_\Omega \left(\frac{|\psi|^{p-1}}{a^{p-2}}\dist(x,\partial\Omega)^{\frac{p(p^\star-q)}{p^\star-p}}\right)^{\frac{p^\star}{q-1}}
\,dx\\
&\leq \left\|\frac{\dist(\cdot,\partial\Omega)^{p+N+q-1-\frac{Nq}{p}}}{a^{p-2}}\right\|^{\frac{p^\star}{q-1}}_{L^\infty(\Omega)}
\int_\Omega\left(\frac{|\psi|}{\dist(x,\partial\Omega)}\right)^{\frac{p^\star(p-1)}{q-1}}\,dx.\end{split}\]
Since $\frac{p^\star(p-1)}{q-1}< p$ when $p^\star -p/(N-p)< q$, we can use the Sobolev embeddings and then the Hardy inequality to conclude case (iii) and, at the same time, the proof of the lemma.
\end{proof}


\begin{proof}[Proof of Theorem \ref{newtheorem:invariance_of_the_cone}]
We approach this proof exactly as we did with the proof of Theorem \ref{theorem:invariance_of_the_cone}. Thus we do not get into all the details and we give a sketch instead. Briefly, we make the choice $w=u-\pi_\alpha(u)$ and we combine Lemma \ref{lemma: additional 1} and Lemma \ref{lemma:invariance_of_cone_third_estimate} to obtain the desired $K$-invariance.
\end{proof}

\smallskip
\begin{remark}
 Due to some technical choices, there exist other variations of the hypotheses from Theorems \ref{theorem:invariance_of_the_cone} and \ref{newtheorem:invariance_of_the_cone} which can be proved similarly. For example, in the particular case $p>N$, condition $\frac{\dist(\cdot,\partial\Omega)}{a(x)^{p-2}}\in L^{p'}(\Omega)$ also guarantees the strict $K$-invariance of the two cones. All these alternatives are increasing the area of possible applications to the multiplicity result provided by Theorem \ref{theorem:not ordered}.
\end{remark}
$$$$$$$$

\noindent{\bf Acknowledgment.}  \noindent The work on this paper started when M.-M. Boureanu was at University Milano Bicocca, on a GNAMPA junior research visit. Such a warm hospitality as the one of the Department of Mathematics and Applications from Milano Bicocca is gratefully acknowledged.


\end{document}